\documentclass[11pt]{amsart}
\usepackage{amssymb}
\usepackage{latexsym}
\usepackage{graphics}
\usepackage{mathptmx}
\usepackage{amsmath}
\usepackage{amsxtra}
\usepackage{amstext}
\usepackage{amssymb}
\usepackage{latexsym}
\usepackage{dsfont}
\usepackage{verbatim}

\DeclareMathOperator{\capa}{cap}

\newcommand{\defeq}{\mathrel{\mathop:}=}

\newcommand{\G}{{\mathbf G}}
\newcommand{\R}{{\mathbb R}}

\newcommand{\sae}{\; \;\sigma-\text{a.e.}}

\newcommand{\mplus}[1]{\mathcal{M}^+({#1})}

\begin{document}
\title[A sublinear version of Schur's lemma and elliptic PDE] {A sublinear version of Schur's lemma and elliptic PDE}
\author{Stephen Quinn}
\address{Department of Mathematics, University of Missouri, Columbia, MO  65211}
\email{stephen.quinn@mail.missouri.edu}
\author{ Igor E. Verbitsky}
\address{Department of Mathematics, University of Missouri, Columbia, MO  65211}
\email{verbitskyi@missouri.edu}

\begin{abstract}
	We study the weighted norm inequality of $(1,q)$-type,
		\[ \Vert \mathbf{G}\nu \Vert_{L^q(\Omega, d\sigma)} \le C \Vert \nu \Vert, \quad \text{ for all } \, \nu \in \mathcal{M}^+(\Omega), \]
		along with its  weak-type analogue, 	for $0 < q < 1$, where $\mathbf{G}$ is an integral operator associated with the nonnegative kernel $G$ on $\Omega\times\Omega$. Here   $\mathcal{M}^+(\Omega)$ denotes the class of positive Radon measures in $\Omega$;  
	$\sigma, \nu \in \mathcal{M}^+(\Omega)$, and $||\nu||=\nu(\Omega)$. 
	
	For both weak-type and strong-type inequalities,  we provide conditions which characterize the measures $\sigma$ for which such an embedding holds.
	The strong-type $(1,q)$-inequality for $0<q<1$ is closely connected with existence  of a positive 
	function $u$ such that $u \ge \mathbf{G}(u^q \sigma)$, i.e., a  supersolution to the integral equation
		\[ u - \mathbf{G}(u^q \sigma) = 0, \quad u \in L^q_{\rm loc} (\Omega, \sigma). \]
	This study is motivated by solving sublinear equations involving the fractional Laplacian,  
		\[ (-\Delta)^{\frac{\alpha}{2}} u - u^q \sigma = 0\]
	in domains $\Omega \subseteq \R^n$ which have a positive Green function $G$, for $0 < \alpha < n$. 
\end{abstract}

\maketitle

\numberwithin{equation}{section}
\newtheorem{theorem}{Theorem}[section]
\newtheorem{lemma}[theorem]{Lemma}
\newtheorem{remark}[theorem]{Remark}
\newtheorem{cor}[theorem]{Corollary}
\newtheorem{prop}[theorem]{Proposition}
\newtheorem{defn}[theorem]{Definition}
\allowdisplaybreaks

\section{Introduction}
	\noindent
	Let $\Omega$ be a locally compact, Hausdorff space, and let $\mathcal{M}^+(\Omega)$ denote the class of all positive Radon measures (locally finite) in $\Omega$.
	For a nonnegative, lower semicontinuous kernel 
	$G\colon \Omega\times \Omega\to [0, +\infty]$, we denote by 
	$$\G \nu(x) =\int_{\Omega} G(x, y) \, d \nu(y), \quad x \in \Omega, $$ 
	 the potential of $\nu \in \mathcal{M}^+(\Omega)$.  
	 
	Let  
	$\sigma \in \mathcal{M}^+(\Omega)$, and let $0 < q < 1$.
	We study the weighted norm inequality
	\begin{align}
		\label{main_strong_inequality}
		\Vert \mathbf{G} \nu \Vert_{L^q(\Omega, \sigma)} \le \varkappa \, \Vert \nu \Vert, \quad \forall \nu \in \mathcal{M}^+(\Omega), 
	\end{align}
	for some positive constant $\varkappa$, where we use the notation $||\nu||=\nu(\Omega)$ if $\nu\in \mathcal{M}^+(\Omega)$ is a finite measure. 
	
	The main goal of this paper is to show that \eqref{main_strong_inequality} is connected to existence of a  measurable function $u$ such that 
	\begin{equation}
			\label{super-sol} 
			u \ge  \G(u^q \sigma), \quad 0<u<+\infty \quad d \sigma-\text{a.e.}  \, \,  \text{ in $\Omega$},  
				\end{equation}
	under certain assumptions on $G$. 
	
	The restrictions on the kernel $G$ studied here include that it satisfies a \textit{weak maximum principle}, 
	and  is \textit{quasi-symmetric} (see the definitions in 
	Sec.~\ref{background} below). 
These restrictions are satisfied by the Green kernel associated with the Laplacian, the fractional Laplacian 
	$(- \Delta)^{\frac{\alpha}{2}}$, and kernels associated with more general  elliptic operators (see \cite{An} and the literature cited there),   as well as radially decreasing convolution kernels 
	$G(x,y) = k(|x-y|)$ on $\mathbb{R}^n$ (\cite{AH}, Sec. 2.6).

	For such  kernels $G$, we show that \eqref{main_strong_inequality} holds if and only if there  exists $u \in L^q(\Omega, \sigma)$ which satisfies \eqref{super-sol}. The additional condition that $u \in L^q(\Omega, \sigma)$ can be dropped using a weighted modification of \eqref{main_strong_inequality} 
	discussed below. 
	
	This equivalence provides a sublinear version of Schur's Lemma for linear integral operators (see 
	\cite{G}). Without the restriction that $G$ satisfies  the weak maximum 
	principle, \eqref{super-sol} with $u \in L^q(\Omega, \sigma)$ does not imply in general that 
	 \eqref{main_strong_inequality}  holds even for positive symmetric kernels $G$. A counter example 
	 is discussed in Sec.~\ref{broken_inequality} below. 
	
	Under further mild assumptions on $G$ (the \textit{non-degeneracy} of the kernel; see 
	Sec.~\ref{background}), we establish that there exists a 
	solution $u \in L^q(\Omega, \sigma)$ to the integral equation
		\begin{equation}
			\label{int_eqn} 
			u - \G(u^q \sigma) =0, \quad 0<u<+\infty \quad d \sigma-\text{a.e.}  \, \,  \text{ in $\Omega$}. 
				\end{equation}
	Such integral equations arise from the study of the sublinear elliptic boundary value problem    
		\begin{equation}
			\label{lap_eqn}
		 \begin{cases}
			- \Delta u - u^q \sigma=0, &u>0 \, \text{ in } \Omega, \\
			u = 0 & \text{ on } \partial \Omega,
		\end{cases} 
		\end{equation}
	where $0 < q < 1$, $\Omega \subseteq \mathbb{R}^n$ is an open domain, and $\sigma \in L^1_{\rm loc}(\Omega)$, or more generally $\sigma \in \mathcal{M}^+(\Omega)$.

	In the following, we will consider the application of our general results to solving the equation involving the fractional Laplacian
		\begin{equation}
			\label{frac_lap_eqn}
			\begin{cases}
				(-\Delta)^{\frac{\alpha}{2}} u - u^q \sigma=0,   & u>0 \, \text{ in } \Omega, \\
				u = 0 & \text{ in $\Omega^c$}.
			\end{cases} 
		\end{equation}
	Note that $(-\Delta)^{\frac{\alpha}{2}}$  is a nonlocal operator 
	for $\alpha \not=2k$ ($k \in \mathbb{N}$), and consequently a condition that $u = 0$ on $\partial \Omega$ is ill-posed.

	If  $(-\Delta)^{{\frac{\alpha}{2}}}$ has a non-negative Green's kernel, then applying the Green's operator $\mathbf{G}$ to both sides, we obtain the equivalent problem \eqref{int_eqn}.

	It is well known  that $G$ satisfies the maximum principle in $\Omega$ in the classical case 
	$\alpha=2$ (Maria \cite{Maria}), and for $0<\alpha \le 2$ (Frostman \cite{Frostman}, see also \cite{F}). 
	For the case $2 < \alpha < n$, we can consider Green's kernels $G$ for nice domains $\Omega\subseteq \R^n$, such as the balls or half-spaces, where the Green's kernel is known to be positive, quasimetrically modifiable, 
	and consequently satisfying the weak maximum principle, which is enough for our purposes. 
	
	In particular, for the entire space $\Omega = \R^n$,  the Green's kernel is the Newtonian kernel if $\alpha=2$, $n \ge 3$, and the Riesz kernel of order $\alpha$ if $0<\alpha <n$. Sublinear equations of the type \eqref{frac_lap_eqn} in this case were 
	treated earlier in \cite{CV1}, \cite{CV2}, \cite{CV3}. 
	
	For the weighted norm inequality \eqref{main_strong_inequality}, we show that it holds if and only if the associated integral equation has a non-trivial supersolution, and actually a solution in a slightly more specific setup.

	\begin{theorem}\label{strong-thm}
		Let $\sigma \in \mathcal{M}^{+}(\Omega)$ and $0<q<1$.  Suppose $G$ is a  lower semicontinuous, quasi-symmetric kernel which satisfies the weak maximum principle.
		Then the following statements are equivalent:
		\begin{enumerate}
			\item There exists a positive constant $\varkappa=\varkappa(\sigma, G)$ such that 
			\begin{align*}
				\Vert  \mathbf{G}\nu \Vert_{L^{q}(\Omega, \sigma)} \le \varkappa \, \Vert \nu \Vert, \quad \forall \,   \nu \in \mathcal{M}^+(\Omega).
			\end{align*}
			\item There exists a supersolution $u \in L^q(\Omega, d \sigma)$ such that \eqref{super-sol}  holds.
			\item There exists a solution $u\in L^q(\Omega, d \sigma)$ to \eqref{int_eqn}  
				provided additionally that $G$ is non-degenerate with respect to $\sigma$.  
		\end{enumerate}
	\end{theorem}
	
	To some degree, the class of measures $\sigma$ for which \eqref{main_strong_inequality} holds, 
	and consequently there is a positive supersolution $u$, 
	 can be understood in terms of energy  norms  of the type 
	\begin{equation}\label{energy-s}
	||\mathbf{G}\sigma ||_{L^s(\Omega, \sigma)}^s = \int_\Omega \left(\mathbf{G}\sigma \right)^s \, d\sigma < +\infty, 
	\end{equation}
	for certain values of $s>0$. This condition with $s=\frac{r}{1-q}$ characterizes  the existence of supersolutions 
	$u \in L^r (\Omega, \sigma)$ satisfying \eqref{super-sol} 
	in the case $r> q$, and is equivalent to the corresponding 
	$(p, r)$-inequality 
	\begin{equation}\label{p-r}
		\Vert \mathbf{G} (f d\sigma) \Vert_{L^r(\Omega, \sigma)} \le C \, \Vert f \Vert_{L^p(\Omega, \sigma)}, 
		\quad \forall f \in L^p(\Omega, \sigma), 
	\end{equation}
	if $0<r<p$ and $p>1$ (see \cite{V16}).
	
	In the case of Riesz potentials on $\Omega=\R^n$, 
	weighted norm inequalities \eqref{p-r} for $0<r<p$ and $p>1$ 
	where studied earlier  in \cite{COV06}, \cite{Maz}, \cite{V99}.

		This study is concerned in a sense with the end-point case of \eqref{p-r} corresponding to $p=1$ and $0<r=q <1$, where it is more natural to use  $\mathcal{M}^{+}(\Omega)$ 
	in place of $L^1(\Omega, \sigma)$ as in \eqref{main_strong_inequality}. We 
	have the following result. 
	
	\begin{theorem}
		\label{energy_norm_results}
		Let $\sigma \in \mathcal{M}^{+}(\Omega)$ and $0<q<1$.  Suppose $G$ is a quasi-symmetric, non-degenerate kernel which satisfies the weak maximum principle.
		\begin{enumerate}
			\item If   \eqref{main_strong_inequality} holds, then $\mathbf{G}\sigma \in L^{\frac{q}{1-q}}(\Omega, \sigma)$.
			\item If $\mathbf{G}\sigma \in L^{\frac{q}{1-q}, q}(\Omega, \sigma)$, then \eqref{main_strong_inequality} holds.
		\end{enumerate}
	\end{theorem}
	Here $L^{s, q}$ is the corresponding Lorentz space (see \cite{SteinWeiss}).

	In Lemma~\ref{lemma2} below, we will show that, without 
	the assumption that $G$ satisfies the weak maximum principle, condition  \eqref{energy-s}  with 
	$s=\frac{q}{1-q}$ 
	is necessary for the existence of a (super)solution $u \in L^q(\Omega, \sigma)$  only if $q\in (0, q_0]$, where 
	 \[q_0 = \frac{\sqrt{5}-1}{2}=0.61\ldots\] 
	 denotes the conjugate golden ratio. In the case $q \in (q_0, 1)$, the optimal value of $s$ in \eqref{energy-s} 
	is $s=1+q$, provided $\sigma$ is a finite measure. For general measures $\sigma$,  
	the existence of a positive  solution $u \in L^q(\Omega, \sigma)$ does not guarantee that 
	\eqref{energy-s} holds if $s= \frac{q}{1-q}$ and $q \in (q_0, 1)$, or $s \not=\frac{q}{1-q}$ for all $q\in(0, 1)$, even for symmetric non-degenerate kernels $G$ (see Sec. \ref{broken_inequality}).

	Another characterization of \eqref{main_strong_inequality} can be deduced from Maurey's results \cite{M} (see also \cite{P}): it  is  equivalent to the existence of a nonnegative function $F \in L^1(\Omega, \sigma)$ which satisfies
		\[ \sup_{y \in \Omega} \int_\Omega  G(x,y) \, F(x)^{1-\frac{1}{q}} \, d\sigma(x) < +\infty. \] 
		This is a dual reformulation of \eqref{main_strong_inequality}, which does not 
		require $G$ to satisfy the weak maximum principle. In the discrete case where $\Omega$ consists 
		of a finite number of points, it represents the duality of the two basic concave programming problems 
		(see \cite{BG}, Sec. 5.7).

	These characterizations have focused on the sublinear case $0< q < 1$.  Note that in the case $q \ge 1$,  obviously  \eqref{main_strong_inequality} holds if and only if  
	\[
	\sup_{y \in \Omega} \, \int_\Omega  G(x,y)^{q}  \, d\sigma(x) < +\infty. 
	\]

	 We also give characterizations of  the weak type $(1, q)$-inequality
		\begin{equation}
			\label{main_weak_inequality}
			\Vert \mathbf{G}\nu \Vert_{L^{q,\infty} (\Omega, d\sigma)} \le C \Vert \nu \Vert, \quad \forall \, \nu \in \mathcal{M}^+(\Omega),
		\end{equation} 
	for any $q>0$, in terms of energy estimates, as well as capacities (see Sec. \ref{weak-section} below).
	Some results of this type were discussed in \cite{QV-Wheeden} under more restrictive 
	assumptions on the kernel $G$, along with analogous characterizations of both strong-type and weak-type $(1, q)$-inequalities involving fractional maximal operators and Carleson measure inequalities for the Poisson kernel.

	In Sec. \ref{QuasimetricSection}, we demonstrate how to remove the extra assumption imposed in 
	Theorem~\ref{strong-thm} 
	that a (super)solution $u \in L^q(\Omega,  \sigma)$ globally. We prove 
	the following theorem where we only assume that $u \in L^q_{{\rm loc}} (\sigma)$, or equivalently, 
	$0<u<+\infty$ $d \sigma$-a.e., provided the kernel $G$ satisfies a weak  form of the complete 
	maximum principle, or alternatively if $G$ is a quasi-metric kernel (see definitions in Sections \ref{background} and \ref{QuasimetricSection}).  

		With a special function $m$ satisfying $0 < m < +\infty$ $d\sigma$-a.e., known as a \textit{modifier} (see, e.g., \cite{FNV}, \cite{HN2012}), we can modify the kernel $G$, so that the modified kernel 
		\begin{equation}
		\label{mod-ker}
		K(x, y) = \frac{G(x, y)}{m(x) \, m(y)}, \quad x, y \in \Omega,
		\end{equation}
		satisfies the weak maximum principle. This makes it possible to apply Theorem~\ref{strong-thm}  
		with $K$ in place of $G$, and consider $u \in L^q_{{\rm loc}} (\sigma)$. A typical modifier 
		that works for general  kernels $G$ which satisfy the complete 
	maximum principle is given by 
	\begin{equation}
		\label{typ-mod}
g(x) = \min \{1, G(x, x_0)\}, \quad x \in \Omega,
	\end{equation}
	where $x_0$ is a fixed pole in $\Omega$ (\cite{HN2012}, Sec. 8).

	\begin{theorem}
		\label{local_solutions-Intro}
		Let $\sigma \in \mplus{\Omega}$ and $0 < q < 1$.
		Suppose $G$ is a quasi-symmetric non-degenerate kernel, continuous in the extended sense on $\Omega\times\Omega$, which either (A) satisfies the complete maximum principle, or (B) is quasimetrically modifiable with modifier 
		given by \eqref{typ-mod}. Then the following statements are equivalent:
		\begin{enumerate}
			\item There exists a positive constant $\varkappa$ such that the weighted norm inequality  
			\begin{equation}
		\label{weight-G}
		\Vert \G \nu \Vert_{L^q(g \, d\sigma)} \le \varkappa \int_\Omega g \, d\nu, \quad \forall \, \nu \in \mplus{\Omega}, 
	\end{equation}
			holds, where the modifier $g(x)$ is given by \eqref{typ-mod} for some $x_0\in \Omega$. \medskip
			
			\item There exists a positive (super)solution $u$   to the equation $u= \mathbf{G} (u^q \, d \sigma)$ such that $u \in L^q_{\rm loc}(\sigma)$ (or equivalently $0<u<+\infty$ $d\sigma$-a.e.) 
		\end{enumerate}
		\end{theorem}

	Theorem~\ref{local_solutions-Intro} yields a characterization of the existence of weak solutions 
	$u \in L^q_{\rm loc}(\sigma)$ 
	to  the fractional Laplacian equation \eqref{frac_lap_eqn} in general domains $\Omega$ with positive Green's function $G$ for $0<\alpha\le 2$, or nice domains (the entire space $\R^n$, or balls or half-spaces in $\R^n$) for $0<\alpha<n$ 
	as discussed above.  In the classical case $\alpha=2$, 
	 such solutions are the so-called very weak solutions to the boundary value problem  	\eqref{lap_eqn} for bounded $C^2$-domains $\Omega$ (see, e.g., \cite{FrazierVerbitsky}, \cite{MV}).

\section{Background on integral kernels}\label{background}

	Let $G\colon X \times Y \rightarrow [0, +\infty]$ be a lower semicontinuous nonnegative kernel, where following the framework of Fuglede \cite{F}, \cite{F65}, we will assume that $X, Y$ are locally compact Hausdorff spaces.
	Every kernel in this paper will be assumed to be of this type, even if not stated explicitly. 
	For most of the following, in particular in the context of strong type $(1, q)$ weighted norm inequalities, we will be working in the case  $X = Y = \Omega$.
	 
	 We denote by $\mathcal{M}^+(X)$ the collection of all nonnegative, locally finite, Borel measures on $X$, and we write $S_\nu$ for the support of $\nu \in \mathcal{M}^+(X)$ and $\Vert \nu \Vert \defeq \nu (X)$ when $\nu$ is a finite measure.

	For $\nu \in \mathcal{M}^+(Y)$, we define the potential of $\nu$ by
		\[ \mathbf{G}\nu (x) \defeq \int_{Y} G(x,y) d\nu(y), \quad \forall x \in X, \]
		and for $\mu \in \mathcal{M}^+(X)$ we have the potential with the adjoint kernel
		\[ \mathbf{G}^*\mu (y) \defeq \int_{X} G(x, y) \, d\mu(x), \quad \forall y \in Y. \]
		
	Let $X = Y = \Omega$, where $\Omega$ is a locally compact Hausdorff space with countable base. 
	The operator $\mathbf{G}$ with kernel $G$ on $\Omega\times \Omega$ 
	is said to satisfy the \textit{Weak Maximum Principle} (with constant $h\ge 1$) provided that  
		\[ \mathbf{G}\nu (x) \le M, \quad \forall x \in S_\nu, \]
		implies
		\[ \mathbf{G}\nu (x) \le h \, M, \quad \forall x \in \Omega, \]
		for any constant $M>0$ and $\nu \in \mathcal{M}^+(\Omega)$. 
	
	When $h=1$, we say that $\mathbf{G}$ satisfies the \textit{Strong Maximum Principle}.
	
	We say that a kernel $G$ satisfies the \textit{Complete Maximum Principle} with constant $h \ge 1$ if,  for any $\mu, \nu \in \mplus{\Omega}$, and constant $c \ge 0$, 
	the inequality 
	\begin{equation}\label{complete-max} 
	 \G\mu (x) \le h \, [\G \nu (x) + c], 
	 \end{equation} 
	 for all $x \in S_\mu$, implies that this inequality holds for all  $x \in \Omega$, provided  $ \G\mu <\infty$ $d \mu$-a.e. 
	This is a form of the \textit{Domination Principle} (see \cite{Doob}, Sec. 1.V.10), which holds for  Green's kernel
	associated with $(-\Delta)^{\frac{\alpha}{2}}$ in the case $0<\alpha\le 2$ with constant $h=1$.

	A kernel $G\colon \Omega \times \Omega \to (0, +\infty]$ is \textit{quasi-symmetric} 
	provided there exists a positive constant $a$ such that
		\[ a^{-1} G(y,x) \le G(x,y) \le a \, G(y,x), \quad \forall x, y \in \Omega. \]

	If $G$ is a quasi-symmetric kernel, note that we can construct a symmetric kernel $G^s$ given by
		\[ G^s(x,y) \defeq G(x,y) + G(y,x) \]
		which is both symmetric and comparable to $G$.
	Indeed, 
		\[ \left(1 + \frac{1}{a}\right)G(y,x) \le G^s(x,y) \le (1 + a) G(y,x), \quad x, y \in \Omega. \]
		We denote the integral operator with kernel $G^s$ by $ \G^s$. 
		
	\begin{remark}
		\label{symmetrized_kernel}
		The inequality
			\begin{align*}
				\Vert \mathbf{G} \nu \Vert_{L^q(\Omega, \sigma)} \le \varkappa \, \Vert \nu \Vert, \quad \forall \nu \in \mathcal{M}^+(\Omega), 
			\end{align*}
		is equivalent to 
			\begin{align*}
				\Vert \G^s \nu \Vert_{L^q(\Omega, \sigma)} \le \varkappa_a \, \Vert \nu \Vert, \quad \forall \nu \in \mathcal{M}^+(\Omega), 
			\end{align*}	
		with only a change in the constant, so that $\varkappa_a$ depends only on $\varkappa$ and $a$. 
		
		Similarly, there is a supersolution $u$ to the inequality
			\[ u \ge \G(u^q \sigma) \]
			if and only if there is a supersolution $u_s$ to the symmetrized inequality
			\[ u_s \ge \G^s(u^q_s \sigma). \]
	\end{remark}
	Indeed, the first equivalence of the remark follows directly from the equivalence of $G$ and $G^s$.  The second equivalence can be shown by scaling $u$ appropriately.

	When $0 < q < 1$, $G$ is a kernel on $\Omega$, and $\sigma \in \mathcal{M}^+(\Omega)$, we are interested in \textit{positive solutions} $u \in L^q(\sigma)$ to the integral equation
		\begin{align}\label{int-eq}
			u = \mathbf{G}(u^q \sigma), \quad u>0 \quad  d\sigma-a.e. \, \, \text{in} \, \, \Omega,
		\end{align}
		and \textit{positive supersolutions} $u \in L^q(\sigma)$ to the integral inequality 
		\begin{align}\label{int-sup}
			u \ge \mathbf{G}(u^q \sigma), \quad u>0 \quad d\sigma-a.e. \, \, \text{in} \, \, \Omega.
		\end{align}
	In Section \ref{QuasimetricSection}, we will discuss how to find solutions $u \in L^q_{\rm loc}(\sigma)$ instead of $u \in L^q(\sigma)$ in the case that the kernels are  quasimetrically modifiable, or satisfy the complete maximum principle. This corresponds to the so-called ``very weak''  
	solutions to the sublinear boundary value problem \eqref{lap_eqn} (see \cite{FrazierVerbitsky}, \cite{MV}). 

	\begin{lemma}
		\label{local_lemma}
		Let $G$ be a lower semicontinuous kernel on $\Omega \times \Omega$, which is non-zero along the diagonal. Let $\sigma \in \mplus{\Omega}$.
		Suppose that \eqref{int-eq} or \eqref{int-sup} holds, where $u<+\infty$ 
		$d\sigma$-a.e. 
		Then $u \in L^q_{\rm loc}(\Omega, \sigma)$.  
	\end{lemma}
	\begin{proof} 
		We consider the case where \eqref{int-sup} holds.
		Let $K \subset \Omega$ be a compact set. Then $K_0=K\cap \{u<+\infty\}$ 
		is a compact set in $\Omega_0=\Omega\cap \{u<+\infty\}$.  
		For each $x \in K_0$, set $c_x \defeq \min\{1, G(x,x)\}$.
		By lower semicontinuity, there exists an open neighborhood $U_x\subset \Omega_0$ 
		such that $G(x,y) > \frac{c_x}{2} > 0$ for $y \in U_x$.
		Since $K_0$ is compact, there exists a finite refinement of the collection $\{ U_x \}$ which covers $K_0$, denoted $\{ U_{x_i} \}_{i = 1}^N$.
		Then
			\begin{align*}
				\int_K u^q \, d\sigma =\int_{K_0} u^q \, d\sigma 
				&\le \sum_{i=1}^N \int_{U_{x_i}} u^q \, d\sigma \\
					&\le \sum_{i=1}^N \frac{2}{c_i} \int_{U_{x_i}} G(x_i, y) u^q(y) \, d\sigma(y) \\
					&\le \sum_{i=1}^N \frac{2}{c_i} u(x_i)< +\infty,  
			\end{align*}
		and thus $u \in L^q_{\rm loc}(\Omega, \sigma)$.
	\end{proof}

	For a measure $\lambda \in \mathcal{M}^+(\Omega)$, the \textit{energy of $\mu$} is given by 
	\[
	\mathcal{E}(\lambda) \defeq \int_\Omega \mathbf{G}\lambda \, d\lambda.
	\]
	The value of the energy of an extremal measure will be shown to be connected with the capacity.
	Following the convention of Fuglede \cite{F}, we say that a kernel 
	$G\colon \Omega\times \Omega \to (-\infty, +\infty]$ is \textit{positive} if $G(x,y) \ge 0$ for every pair $(x,y) \in \Omega \times \Omega$.
	A kernel $G$ is \textit{strictly positive} if $G$ is positive and additionally $G(x,x) > 0$ for every $x \in \Omega$.
	We say a kernel is \textit{pseudo-positive} if $\mathcal{E}(\mu) \ge 0$ for every measure $\mu \in \mathcal{M}^+(\Omega)$ with compact support.
	A kernel is \textit{strictly pseudo-positive} if $\mathcal{E}(\mu) > 0$ for every $\mu \neq 0$, $\mu \in \mathcal{M}^+(\Omega)$ with compact support.
	A positive kernel is obviously pseudo-positive, and a kernel is strictly positive if and only if it is strictly 
	pseudo-positive (\cite{F}, p. 150).
	
	The kernel $G$ is said to be \textit{degenerate} with respect to $\sigma \in \mathcal{M}^+(\Omega)$ provided there exists a set $A \subset \Omega$ with $\sigma(A) > 0$ and
		\[ G(\cdot, y) = 0 \quad \text{$d\sigma$-a.e. for $y \in A$.} \]

	Otherwise, we will say that $G$ is \textit{non-degenerate} with respect to $\sigma$. (The notion of non-degeneracy appeared in special conditions in \cite{Sinnamon} in the context of $(p, q)$-inequalities for positive operators 
		$T\colon L^p \to L^q$ in the case $1<q \le p < +\infty$.) We will sometimes  rule out degenerate kernels from study since the corresponding integral equations \eqref{int_eqn} cannot have positive solutions.

\section{Modified kernels and $L^q_{\rm loc}(\sigma)$ solutions}\label{QuasimetricSection}
	In this section, we wish to describe how to find local solutions $u \in L^q_{\rm loc}(\sigma)$ to the equation
	\begin{equation}
	\label{int-eq-sec3}
		\begin{cases}
			u = \G(u^q \sigma) \quad d \sigma-\text{a.e.} \, \, \text{in} \, \, \Omega, \\
			u \in L^q_{\rm loc}(\sigma), 
		\end{cases}
	\end{equation}
	from global solutions $v \in L^q(\omega)=L^q(\Omega, \omega)$ to the equation
	\begin{equation}
		\label{mod-int-eq}
		\begin{cases}
			v = \mathbf{K} (v^q \omega) \quad d \omega-\text{a.e.} \, \, \text{in} \, \, \Omega, \\
			v \in L^q(\omega).
		\end{cases}
	\end{equation}
	
	Here, $K$ is the modified kernel \eqref{mod-ker} with modifier \eqref{typ-mod} denoted by 
	\[ g(x) = \min\{1, G(x, x_0)\}, \]
	 where $x_0\in \Omega$ is a fixed pole, $v \defeq \frac{u}{g}$, and $d\omega \defeq g(x)^{1+q} \, d\sigma$.
	
	In this case, we introduce the relevant $(1, q)$-weighted norm inequalities for this section:
	\begin{equation}
		\label{weighted-G}
		\Vert \G \nu \Vert_{L^q(g \, d\sigma)} \le \varkappa \, \int_\Omega g \, d\nu \quad \text{for all } \nu \in \mplus{\Omega}, 
	\end{equation}
	and
	\begin{equation}
		\label{main-int-K}
		\Vert \mathbf{K} \nu \, \Vert_{L^q(\omega)} \le \varkappa \, \Vert \nu \Vert \quad \text{for all } \nu \in \mplus{\Omega}.
	\end{equation}
	Note that \eqref{main-int-K} is simply \eqref{weighted-G} restated with $\mathbf{K}$ and $\omega$ 
	in place of $\mathbf{G}$ and $\sigma$.

	In this section, we consider two classes of kernels--quasimetrically modifiable kernels and kernels satisfying the Complete Maximum Principle--and show that if these kernels are modified, the modified kernels then satisfy the Weak Maximum Principle and thus Theorem~\ref{strong-thm} applies when \eqref{main_strong_inequality} holds with $\mathbf{K}$ and $\omega$ in place of $\G$ and $\sigma$.
	For domains $\Omega\subset\R^n$ satisfying the boundary Harnack principle, such as bounded Lipschitz domains and NTA domains, the Green's kernels $G$ for the Laplacian and fractional Laplacian (in the case $0<\alpha\le 2$) are quasimetrically modifiable.
	Examples of quasimetric kernels and quasimetrically modifiable kernels can be found in \cite{FNV}.

	We say that $d(x,y)\colon \,  \Omega \times \Omega \rightarrow [0, + \infty)$ satisfies the quasimetric triangle inequality with quasimetric constant $\kappa>0$ provided
		\begin{equation}\label{quasitr} 
			d(x,y) \le \kappa [d(x,z) + d(z, y)], 
		\end{equation}
		for any $x, y, z \in \Omega$, and $d(x, y)\not=0$ for some $x, y\in \Omega$. 
	Without loss of generality we may assume $\kappa \ge \frac{1}{2}$. We say that $G$ is a \textit{quasimetric} kernel with quasimetric constant $\kappa$ provided $G$ is symmetric and $d(x,y) \defeq \frac{1}{G(x,y)}$ satisfies \eqref{quasitr}.  
	
	We say the kernel $G$ is \textit{quasimetrically modifiable} with constant $\kappa$ if there exists a measurable function $m: \Omega \rightarrow (0, +\infty)$, called a modifier, such that
		\begin{equation}
			K(x,y) \defeq \frac{G(x,y)}{m(x)m(y)}
		\end{equation}
		defines a quasimetric kernel with quasimetric constant $\kappa$.
		
		\begin{remark} The two modifiers we will primarily work with are $G^{x_0}(x)\defeq G(x, x_0)$ and $g(x)\defeq \min \{ 1, G^{x_0}(x) \}$ for some fixed pole $x_0 \in \Omega$.
	Further development and discussion of quasimetric kernels can be found in \cite{FNV}, \cite{H}, \cite{HN2012}, \cite{KV}.
	\end{remark}

		\begin{remark}
			Since we wish to apply our existence theorems for supersolutions to the modified kernel $K$, we will sometimes require additionally that either $G(x,y)$ is continuous off the diagonal, or continuous 
			on $\Omega\times \Omega$ in the extended sense, so that $K(x,y)$ will be lower semicontinuous.
		\end{remark}
		
	We recall the so-called Ptolemy's inequality for quasimetric spaces \cite{FNV}: If $d$ is a quasimetric with constant $\kappa$ on $\Omega$, then 
		\begin{equation}\label{ptolemy} 
	 		d(x,z)d(y,w) \le 4 \kappa^2 \Big[d(x,y)d(z,w) + d(y, z)d(x,w)\Big], 
	 	\end{equation}
		for any $w, x, y, z \in \Omega$.  The following lemma is immediate from \eqref{ptolemy} 
		(see also \cite{HN2012}, Proposition 8.1 and Corollary 8.2). 
	
	\begin{lemma}
		If $G$ is a quasimetric kernel on $\Omega$ with quasimetric constant $\kappa$, then 
		\[K(x,y)= \frac{G(x,y)}{G^{x_0}(x) G^{x_0}(y)} \] 
		is a quasimetric kernel on $\Omega \setminus \{ x\colon \, G(x,x_0) = +\infty \}$ with quasimetric constant $4 \kappa^2$.
	\end{lemma}

		We will need an analogous statement for modifiers $g$ in place of 
		$G^{x_0}$. (See \cite{HN2012}, Corollary 8.4, where a similar result 
		is proved for Green's functions associated with a Brelot space.) 
	
	\begin{lemma}\label{mod-lemma} Let $x_0\in \Omega$, and let $g(x) = \min\{ 1, G(x, x_0)\}$. 
		If $G$ is a quasimetric kernel on $\Omega$ with quasimetric constant $\kappa$, then \[K(x,y) = \frac{G(x,y)}{g(x)g(y)}\] is a quasimetric kernel on $\Omega \setminus \{ x : G(x,x_0) = +\infty \}$ 
	 with quasimetric constant 	$4 \kappa^2$.
	\end{lemma}
		\begin{proof}
			By \eqref{ptolemy}, we have 
			\[ \frac{1}{G(x,y)} \frac{1}{G(z,x_0)} \le 4 \kappa^2 \left[ \frac{1}{G(x,z)} \frac{1}{G(y,x_0)} + \frac{1}{G(x,x_0)} \frac{1}{G(z,y)} \right], \]
			from which it follows that
			\[ \frac{g(x)g(y)}{G(x,y)} \le 4 \kappa^2 \left[ \frac{g(x)}{G(x,z)} + \frac{g(y)}{G(z,y)} \right] G(z, x_0). \]
			Now we wish to consider several cases in order to replace $G(z, x_0)$ with $g(z)$.  If $G(z, x_0) \le 1$, then we are done.
			We focus on the case where $G(z, x_0) > 1$, which implies $g(z) = 1$.
			
			First, consider the subcase where $G(y, x_0)> 1$ and $G(x, x_0) > 1$.  Then $g(x) = g(y) = 1$ and our desired result is precisely the quasimetric triangle inequality for $G$.
			
			We now consider the case where $G(y, x_0) < 1$ and $G(y, x_0) \le G(x, x_0)$ (the case $G(x, x_0) < 1$ and $G(x, x_0) \le G(y, x_0)$ is similar).
			In this case, $g(y) = G(y, x_0)$ and $g(y) \le g(x)$.
			This reduces to showing 
				\[ \frac{g(x)g(y)}{G(x,y)} \le 4\kappa^2 \left[ \frac{g(x)}{G(x,z)} + \frac{g(y)}{G(y,z)} \right]. \]
			Since $g(x) \le 1$, using the quasimetric triangle inequality for $d(x,y)$, we deduce 
			\[ \frac{g(x)g(y)}{G(x,y)} \le \frac{g(y)}{G(x,y)} \le \kappa \left[ \frac{g(y)}{G(x,z)} + \frac{g(y)}{G(y,z)} \right] \le 4 \kappa^2 \left[ \frac{g(x)}{G(x,z)} + \frac{g(y)}{G(y,z)} \right], \]
			which is the desired inequality.
		\end{proof}
		Note that, under the assumptions of Lemma~\ref{mod-lemma},  when $G$ is finite off the diagonal, then $K$ is a quasimetric kernel on the punctured domain $\Omega \setminus \{x_0\}$.
		
	\begin{lemma}
		\label{qmm-wmp}
		Let $K$ be a quasimetric kernel with quasimetric constant $\kappa$.
		Then $K$ satisfies the Weak Maximum Principle with constant $h = 2\kappa$.  
	\end{lemma}
		\begin{proof}
			For $x, y \in \Omega$, let $d(x,y) = \frac{1}{K(x,y)}$.
			Suppose $\mu \in \mplus{\Omega}$ and $\mathbf{K} \mu(x) \le 1$ on $S_\mu$, where we may assume without loss of generality that $S_\mu$ is a compact set in $\Omega$. 
			Suppose $x \in \Omega \setminus S_\mu$.
			Let $x' \in S_\mu$ be a point which ``minimizes'' (up to an $\epsilon>0$) the quasi-distance between $x$ and $S_\mu$.
			For all $y \in S_\mu$, note that
				\[ d(y, x') \le \kappa \left[ d(y, x) + d(x', x) \right] \le (2 \kappa+\epsilon) \,  d(x,y). \]
			This implies that $K(x,y) \le (2\kappa + \epsilon) \, K(x', y)$, and consequently 
			$$\mathbf{K} \mu(x) \le (2\kappa + \epsilon) \, \mathbf{K}\mu(x') \le 2 \kappa + \epsilon.$$ 
			Letting $\epsilon\to 0$, we deduce that  $K$ satisfies the Weak Maximum Principle with constant $h = 2\kappa$.
		\end{proof}
	
	\begin{lemma}
		\label{totalmax-wmp}
		Let $G$ be a positive kernel on $\Omega$ and let $K$ be the modified kernel 
			\[ K(x,y) = \frac{G(x,y)}{g(x)g(y)}. \]
		If $G$ satisfies the Complete Maximum Principle \eqref{complete-max} with constant $h\ge 	1$, then $K$ satisfies the Weak Maximum Principle with the same constant.
	\end{lemma}
	\begin{proof}
		Let $\mu \in \mplus{\Omega}$.
		First, we claim that $d \nu \defeq \frac{d \mu}{g} \in \mplus{\Omega}$.
		Let $F \subset \Omega$ be a compact set.
		By lower semicontinuity of $g$, it follows that $1 \ge g(x) \ge c > 0$ on $F$, and so $\nu(F) \le \frac{1}{c} \mu(F)$.
		This shows that $\nu$ is locally finite, and $S_\mu=S_\nu$. 
		
		Now suppose $\mathbf{K} \mu \le 1$ on $S_\mu$.
		We wish to show that $\mathbf{K} \mu \le h$ on $\Omega$.
		Notice that $\G \nu \le g(x)$ on $S_\nu$, where $d \nu = \frac{d\mu}{g}$. Consequently, 
		 $\G \nu \le 1$ and $\G\nu \le \G\delta_{x_0}$ on $S_\nu$. 
		By the Complete Maximum Principle with constant $h\ge 1$, 
		it follows that $\G \nu \le h$ on $\Omega$, 
		and at the same time  $\G \nu \le h \, \G\delta_{x_0}$ on $\Omega$.
		Hence, $\G \nu \le h \, g(x)$ on $\Omega$.
		Converting our expression back to terms of $\mathbf{K}$ and $\mu$ proves the claim.
	\end{proof}

	We are now ready to prove Theorem~\ref{local_solutions-Intro}.
	\begin{proof}[Proof of Theorem~\ref{local_solutions-Intro}]
			Let $d \omega = g^{1+q} \, d \sigma$. It is easy to see by definition of $\G$, $\mathbf{K}$, $u$, $v,$ and $\omega$ that \eqref{weighted-G} and \eqref{main-int-K} are equivalent.
		If \eqref{weighted-G} holds, then, by Theorem~\ref{strong-thm}, there exists a solution $v \in L^q(\omega)$ to \eqref{mod-int-eq}.  Then by Lemma~\ref{local_lemma}, we have $u \defeq g v \in L^q_{loc}( \sigma)$, and $u$ is a solution to \eqref{int-eq-sec3}.
		
		Conversely, suppose $u \in L^q_{loc}(\sigma)$ is a supersolution to \eqref{int-eq-sec3}. Note that $v \in L^q(\omega)$ if and only if
			\begin{equation}
				\int_\Omega u(x)^q  g(x) \, d\sigma(x) < + \infty 
			\end{equation}
			holds.
			Since $u \in L^q_{loc}(\sigma)$, we have $u(x) < +\infty$ $d \sigma$-a.e.  
			Further, 
				\[
					\int_\Omega g(x) \, u(x)^q d\sigma(x) \le \int_\Omega G(x, x_0) \, u(x)^q d\sigma(x) \le  u(x_0),
				\]
			which establishes that $v \in L^q(\omega)$ provided $u(x_0)<+\infty$. 
		Since (A) or (B) holds, by Lemma~\ref{qmm-wmp} and Lemma~\ref{totalmax-wmp} 
		it follows that $K$ satisfies the weak maximum principle.
		Therefore, by Theorem~\ref{strong-thm}, inequality \eqref{main-int-K} holds and so \eqref{weighted-G} holds as well.
		\end{proof}
	
	\begin{remark}
	It follows from the proof of Theorem~\ref{local_solutions-Intro} that statement (1) holds for 
	the weight $g(x) = \min \{G(x, x_0), 1\}$ with any $x_0 \in \Omega$ provided 
	 $u(x_0) < + \infty$, where $u$ is the supersolution in statement (2). 
	 Consequently, if statement (1) holds for at least one $x_0 \in \Omega$, then it holds 
	 for every $x_0 \in \Omega$, except possibly a set of $\sigma$-measure zero. In the case of Green's kernel 
	 associated with $(-\Delta)^{\frac{\alpha}{2}}$ in the case $0<\alpha\le 2$ it is easy to see that we can use any $x_0\in \Omega$, since otherwise $u\equiv +\infty$ in $\Omega$. 
	\end{remark} 
		
\section{Summary of potential theory}
	\label{capacity_theory}
	
	A major tool in the proofs of both the strong type and weak type results will be the notions of capacity of a set and the associated equilibrium measure.
	We will start by describing potentials of kernels on $X \times Y$ used in the context of weak type inequalities; then we will narrow our focus to kernels on $\Omega \times \Omega$ in the case 
	$X=Y=\Omega$ having in mind applications to strong type counterparts.

	 For a kernel $G \colon X\times Y \to [0, +\infty]$, we 
	 will be using several related notions of capacity. Let $K \subset X$ be a compact set. 
	The initial two capacities we consider:
	\begin{align}
		\capa_0 (K) &\defeq \sup \{ \mu(K)\colon \, \, \mu \in \mathcal{M}^+(K), \quad \mathbf{G}^*\mu(y) \le 1 \, \, \text{ for all } y \in Y \}, \label{cap_0-def}\\
		\operatorname{cont} (K) &\defeq \inf \{ \lambda(Y)\colon \, \, \lambda \in \mathcal{M}^+(Y), \quad \mathbf{G}\lambda(x) \ge 1 \, \, \text{ for all } x \in K \}, 
		\label{cont-def}
	\end{align}
	are discussed by Fuglede \cite{F65} and Brelot \cite{Brelot}.
	 
	In fact, Fuglede \cite{F65} showed that these two notions of capacity (content)  coincide with the use of von Neumann's Minimax Theorem. 
	The study of capacities provides characterizations of weak-type inequalities like \eqref{main_weak_inequality}, as we will see in Section~\ref{weak-section}.

	In the case $G\colon \Omega \times \Omega \to [0, +\infty]$, we consider the \textit{Wiener capacity} 
		\begin{align*}
			\capa_1 (K) &\defeq \sup \{ \mu(K)\colon \mu \in \mathcal{M}^+(K), \quad \mathbf{G}^*\mu(y) \le 1 \text{ for all } y \in S_\mu \}, 	
		\end{align*}
		for compact sets $K \subset \Omega$.

	The extremal measure $\mu$ which attains the capacity will be referred to as the \textit{equilibrium measure}; it exists under certain assumptions on $G$ (see 
	Theorem \ref{fuglede_thm} below).

	Unless otherwise noted, we will work with this capacity.
	Note that $\capa_0(K) \le \capa_1(K)$, and in the case where $\mathbf{G}$ satisfies the weak maximum principle we have $\capa_1(K) \le h \capa_0(K)$.
	Capacity can also be computed via an extremal energy problem:
		\[ \capa_1(K) = \left( w[K] \right)^{-1} \]
		where \[ w[K] \defeq \inf \{ \mathcal{E}(\mu)\colon \, \, \, \mu \in \mathcal{M}^+(K), \quad \mu(K) = 1 \}. \]
	
	We say that a property holds \textit{nearly everywhere} (or n.e.) on $K$ when the exceptional set $Z \subset K$ has capacity $\capa_1(Z) = 0$.
	The following lemmas will help us to work with sets of zero capacity.
	\begin{lemma}\label{cap_zero_lemma}
		If $\mu \in \mathcal{M}^+(K)$, $\mu \not \equiv 0$, and $\capa_1 (K) = 0$, then $\mathbf{G}^*\mu = +\infty$ $d\mu-$a.e in $K$.
	\end{lemma}
	\begin{proof}
		Set 
			\[E= \{ x \in K : \, \, \mathbf{G}^*  \mu (x)< +\infty\}.\]
		Notice that $E = \bigcup_{n=1}^{\infty} F_n$, where
			$F_n = \{x \in K\colon \, \, \mathbf{G}^*  \mu(x) \le n\}$
		is a closed set by the lower semicontinuity of $G$, and consequently is a compact subset of $K$.
		In particular,  $E$ is a Borel set. 
		
		Suppose that $\capa_1 (K) =0$.
		Then $\capa_1 (F_n) =0$, and hence $\mu(F_n)=0$, for every $n=1, 2, \ldots$, in view of the definition of $\capa_1 (F_n)$. 
		It follows that 
		\[ \mu(E) \le \sum_{n=1}^\infty \mu(F_n) = 0. \]
		This proves that $\mathbf{G}^* \mu = +\infty$ $d \mu$-a.e. on $K$. 
		\end{proof}
	
	\begin{lemma}\label{soln_abs_cont}
		Let $q > 0$. Suppose $\sigma \in \mathcal{M}^+(K)$, and $\mathbf{G}^*(u^q \sigma) \le u$ $d\sigma$-a.e., where $\int_K u^q \, d\sigma < + \infty$ for every compact set $K \subset \Omega$.
		Then $d\omega \defeq u^q d\sigma$ is absolutely continuous with respect to capacity, i.e., $\capa_1(K) = 0$ yields $\omega(K) = 0$.
		If in addition $u > 0$ $d\sigma$-a.e. on $K$, where $\capa_1(K) = 0$, then $\sigma(K) = 0$.
	\end{lemma}
	\begin{proof}
		Suppose $K$ is a compact set subset of $\Omega$. 
		Since 
			\[\mathbf{G}^*\omega \le u \quad d \sigma\rm{-a.e.}, \]
		we deduce 
			\[ \int_{K} (\mathbf{G}^*  \omega)^q \, d \sigma \le \int_K u^q d \sigma = \omega(K) < \infty.\]
		Hence $\sigma(\{x \in K: \mathbf{G}^*  \omega = +\infty\}) =0$.
		Since $\omega$ is absolutely continuous with respect to $\sigma$, it follows that  $\omega(\{x \in K: \mathbf{G}^*  \omega = +\infty\}) =0$. 
		If  $\capa_1 (K) =0$, then by the previous lemma  $\omega(K)=0$. 
		This yields $\sigma(K)=0$, unless $u=0$ $d \sigma$-a.e. on $K$. 
	\end{proof}

	The following result of Fuglede~\cite{F} will be important in deriving inequality \eqref{main_strong_inequality} from a known positive supersolution for \eqref{int-sup}.
	
	\begin{theorem}\label{fuglede_thm} Let $G$ denote a symmetric, pseudo-positive kernel, and $K$ a compact set with $\capa_1 K < + \infty$.  The two maxima problems
		\begin{align*}
			&\lambda( K ) = \text{\textnormal{maximum}} \quad  (\text{\textnormal{where}} \, \,   \lambda \in \mathcal{M}^+(K), \; \mathbf{G}\lambda \le 1 \, \, \text{\textnormal{on}} \, \, S_\lambda), \\
			&2 \lambda (K) - \mathcal{E}(\lambda) = \text{\textnormal{maximum}} \quad (\text{\textnormal{where}} \, \, \lambda \in \mathcal{M}^+(K)), 
		\end{align*}
		have precisely the same solutions, and the value of each of the two maxima is the Wiener capacity $\capa_1 K$.
		The class of all solutions is compact in the vague topology on $\mathcal{M}^+$ and consists of all measures $\lambda \in \mathcal{M}^+(K)$ for which \[ \mathcal{E}(\lambda) = \lambda(\Omega) = \capa_1 K. \]
		The potential of any solution has the following properties:
		\begin{enumerate}
			\item $\mathbf{G}\lambda(x) \ge 1 $ \text{\textnormal{nearly everywhere in}} $K$, 
			\item $\mathbf{G}\lambda(x) \le 1$ \text{\textnormal{on}} $S_\lambda$, 
			\item $\mathbf{G}\lambda(x) = 1$ $d \lambda\text{\textnormal{-a.e. in}}$ $\Omega$.
		\end{enumerate}
	\end{theorem}
	
		Note that the extremal measure $\lambda$ in Theorem~\ref{fuglede_thm} is equilibrium measure for the set $K$.
We observe that the previous theorem requires that the capacity of the compact set $K$ to be finite.  To deal with this requirement, we will make sure that the kernel is strictly pseudo-positive.
	\begin{remark} \label{finite_cap}
		Let $G$ be a kernel on $\Omega$.
		Then $\capa_1 K < + \infty$ for every compact $K \subset \Omega$ if and only if $G$ is strictly pseudo-positive.
	\end{remark}
		Indeed (see \cite{F}, p. 162), since $K$ is compact, the minimization problem 
		\[ 
		w(K) = \inf \mathcal{E}(\mu),\] 
		taken over all unit measures $\mu \in \mathcal{M}^+(K)$, attains its minimum.
		Therefore, $w(K) > 0$ by the strict pseudo-positivity of the kernel, and therefore $\capa_1(K) = 
		\frac{1}{w(K)} < +\infty$.
		
		Conversely, if $\capa_1 K < + \infty$ for every compact $K \subset \Omega$, then for each $x_0 \in \Omega$, we see that the point mass $\delta_{x_0}$ is the extremal measure for $w(K)$, with $G(x_0, x_0) = w(\{x_0\}) = \frac{1}{\capa_1 K} > 0$.
		This shows that the kernel is strictly positive, and therefore is strictly pseudo-positive.

\section{Proof of strong type results}
	The proof of Theorem \ref{strong-thm} is broken in parts contained within the following subsections.
	As shown in Section \ref{QuasimetricSection}, we can find solutions $u \in L^q_{\rm loc}(\sigma)$ by passing to a modified kernel and determining solutions $v \in L^q(\omega)$.
	Going from the inequality \eqref{main_strong_inequality} to supersolution \eqref{int-sup} follows from  a lemma due to Gagliardo \cite{G, Szeptycki} and does not require $G$ to be quasi-symmetric or to satisfy the weak maximum principle.
	However, the converse statement does not hold without the weak maximum principle.
	Indeed, we provide an example of such a kernel in Section \ref{broken_inequality}.
		
	\begin{proof}[Proof of Theorem 1.1]
		That (1)$\Longrightarrow$(2) follows from Lemma \ref{phi_soln} and Remark \ref{phi_rmk}.
		That (2)$\Longrightarrow$(3) follows from Lemma \ref{supsoln_to_soln}. The implication (3)$\Longrightarrow$(2) is trivial, and  (2)$\Longrightarrow$(1) follows from Lemma \ref{supsoln_impl_ineq}.
	\end{proof}
\subsection{Energy estimates}
	Important to our study of the strong type inequality \eqref{main_strong_inequality} are energy conditions of the type 
		\begin{equation} \label{energy}
			\int_{\Omega} (\G \sigma)^s d \sigma< \infty,  
		\end{equation}
		for some $s>0$. 
	Note that when $s = 1$, we are computing the energy $\mathcal{E}(\sigma)$ introduced above.
	We first start with providing a proof of Theorem \ref{energy_norm_results}.
	\begin{proof}[Proof of Theorem \ref{energy_norm_results}]
		(1)
		First, suppose that the strong type inequality \eqref{main_strong_inequality} holds, where $G$ is  a quasi-symmetric kernel with quasi-symmetry constant $a$. (Notice that the weak maximum principle is not used in the proof of this statement.)  
		By Maurey's theorem \cite{M}, \eqref{main_strong_inequality} yields the existence of a nonnegative function $F\in L^1(\sigma)$, $F>0$ $d\sigma-$a.e., so that $||\G^{*}(F^{1-\frac{1}{q}} d \sigma)||_{L^\infty(d \sigma)}\le 1$. 
		Hence, by quasi-symmetry of $G$ it follows that $||\G(F^{1-\frac{1}{q}} d \sigma)||_{L^\infty(d \sigma)}\le a$,  
		and by H\"older's inequality with exponents $\frac{1}{q}$ and $\frac{1}{1-q}$, we deduce 
		\begin{equation*}
			\begin{aligned} 
				\G \sigma (x) 
					&= \int_{\Omega} F(y)^{-1} G(x, y) F(y) d \sigma(y)\\
					& \le \left[ \G(F^{1-\frac{1}{q}} d\sigma)(x)\right]^{q} 
				\left[\G(Fd \sigma)(x)\right]^{1-q} \\&\le a^q \left[\G(Fd \sigma)(x)\right]^{1-q} \quad d \sigma-\text{a.e.}
			 \end{aligned}
		\end{equation*}
		Using the preceding inequality, H\"older's inequality, and Fubini's theorem, we estimate 
		\begin{equation*}
			\begin{aligned} 
				\int_\Omega (\G \sigma)^{\frac{q}{1-q}} d \sigma & \le a^{\frac{q^2}{1-q}} \int_{\Omega} \left[\G(Fd \sigma)\right]^{q} F^{-1} 
				F d \sigma \\& \le a^{\frac{q^2}{1-q}} \left[\int_{\Omega} \G(Fd \sigma)  F^{1-\frac{1}{q}} d \sigma \right]^q ||F||^{1-q}_{L^1(\sigma)}\\
				& = a^{\frac{q^2}{1-q}}\left[ \int_{\Omega} \G^{*}(F^{1-\frac{1}{q}} d \sigma) F d \sigma\right]^q ||F||^{1-q}_{L^1(\sigma)} \le  a^{\frac{q^2}{1-q}} \, ||F||_{L^1(d \sigma)}<\infty.
			 \end{aligned}
		 \end{equation*}
		 Thus we have established $\mathbf{G}\sigma \in L^{\frac{q}{1-q}}(\sigma)$ provided  \eqref{main_strong_inequality}  holds for  $q\in(0, 1)$.
		 
		(2) Now, suppose that $\mathbf{G}\sigma \in L^{\frac{q}{1-q}, q}(\sigma)$. 
		We note that $\sigma$ is absolutely continuous with respect to capacity.
		Indeed, suppose this were not the case, then by Lemma \ref{cap_zero_lemma}, $\mathbf{G}\sigma = +\infty$ on a set of positive $\sigma$ measure.
		This contradicts the finiteness of $\Vert \mathbf{G}\sigma\Vert_{L^{\frac{q}{1-q}, q}(\sigma)} < + \infty$.
		By non-degeneracy, we know $\mathbf{G}\sigma \not \equiv 0$ on a set of positive $\sigma$ measure, and hence division by $\mathbf{G}\sigma$ is well defined.
		By duality we find
		\begin{align*}
			\Vert \mathbf{G}\nu \Vert^q_{L^q(\sigma)}
				&= 	\left\Vert \left( \frac{\mathbf{G}\nu}{\mathbf{G}\sigma} \right)^q (\mathbf{G}\sigma)^q \right\Vert_{L^1(\sigma)} \\
				&\le \left\Vert \left( \frac{\mathbf{G}\nu}{\mathbf{G}\sigma} \right)^q \right\Vert_{L^{\frac{1}{q}, \infty}(\sigma)} \Vert (\mathbf{G}\sigma)^q \Vert_{L^{\frac{1}{1-q}, 1}(\sigma)} \\
				&= \left\Vert \frac{\mathbf{G}\nu}{\mathbf{G}\sigma} \right\Vert^q_{L^{1, \infty}(\sigma)} \Vert \mathbf{G}\sigma \Vert^q_{L^{\frac{q}{1-q}, q}(\sigma)} \\
				&\le C \Vert \nu \Vert^q, 
		\end{align*}
		where the final inequality follows from Lemma \ref{weak_frac_bound}.
		Hence we have established the strong type inequality \eqref{main_strong_inequality}. 
	\end{proof} 
	
	As the above proof shows, the energy condition is closely related to the existence of the strong type inequality.
	The following lemma shows that knowing only that a supersolution exists allows us to obtain similar energy estimates.
	These estimates will allow us later to construct solutions to our integral equation from supersolutions.

	In the next lemma, we deduce \eqref{energy} for various values of $s$ without assuming that \eqref{main_strong_inequality} holds, and without using  
	the weak maximum principle,  for general quasi-symmetric kernels $G$. 
	
	Let $q_0 = \frac{\sqrt{5}-1}{2}=0.61\ldots$ denote the conjugate golden ratio. 
	
	\begin{lemma}\label{lemma2} Suppose $G$ is a  quasi-symmetric kernel on $\Omega\times \Omega$ 
	with quasi-symmetry constant $a$. 
		Let $\sigma \in \mathcal{M}^+(\Omega)$.  Suppose there is a positive supersolution $u \in L^q(\Omega, \sigma)$ to \eqref{int-sup}.
		
		(a) Let $0 < q \le q_0$.  Then \eqref{energy} holds with $s=\frac{q}{1-q}$, and 
			\begin{equation} \label{est1}
			\int_{\Omega} (\G \sigma)^\frac{q}{1-q} d \sigma \le c \, \int_\Omega u^{q} d \sigma,   
			\end{equation}
			where $c=a^{\frac{q^2}{1-q}}$.
	
		(b) If $q_0 < q <1$, and $\sigma$ is a finite measure, then \eqref{energy} holds for $0<s\le 1+q$, and 
			\begin{equation} \label{est2}
			\int_{\Omega} (\G \sigma)^{s} d \sigma   \le c \, \left[\int_\Omega u^{q} d \sigma\right]^{\frac{s(1-q)}{q}} \left[\sigma(\Omega)\right]^{1-\frac{s(1-q)}{q}},     
			\end{equation}
			where $c=a^{\frac{s}{1+q}}$.
			
			For symmetric kernels $G$, both  (5.2) and (5.3) hold with $c=1$.
	\end{lemma}
	\begin{remark}
		For $q_0 < q <1$, statement (a) generally fails. More precisely, there exists a strictly positive symmetric kernel $G$ and measure $\sigma$ such that there is a positive solution $u \in L^q(\Omega, \sigma)$ to \eqref{int-eq}, but $\int_\Omega (\G \sigma)^{\frac{q}{1-q}} d \sigma=+\infty$; see 
		Sec. \ref{broken_inequality}. 
	\end{remark}
	\begin{remark}
		The exponents $s=\frac{q}{1-q}$ and $s=1+q$ in statements (a) and (b) respectively are sharp, i.e., there exist symmetric kernels $G$ for which \eqref{energy} fails if $s\not=\frac{q}{1-q}$ in the case of general measures $\sigma$, and if  $s >\min \left(\frac{q}{1-q},1+q\right)$ in the case of finite measures $\sigma$; see 
		Sec. \ref{broken_inequality} below and \cite{V16}.
	\end{remark}	

	\begin{proof}
		Suppose $u$ is a positive supersolution satisfying \eqref{int-sup}.
		Suppose \[q\le s\le \min \left( \frac{q}{1-q}, 1+q \right).\]
		Let $r=\frac{s}{q}$. 
		By H\"older's inequality with exponents $r$ and $r'=\frac{r}{r-1}$, 
		\begin{equation*}
			\begin{aligned} 
			\G \sigma (x) & = \int_{\Omega}  u^{\frac{q}{r}}  u^{-\frac{q}{r}} G(x, y) \, d \sigma(y)\\& \le \left[ \G(u^q 
			d \sigma(x)\right]^{\frac{1}{r}}  \left[\G(u^{-\frac{q}{r-1}} d \sigma)(x)\right]^{\frac{1}{r'}} \\&
			\le [u(x)]^{\frac{1}{r}}  \left[\G(u^{-\frac{q}{r-1}} d \sigma)(x)\right]^{\frac{1}{r'}} . 
			\end{aligned}
		\end{equation*}
		Suppose $r'\ge s$. Using the preceding inequality, H\"older's inequality with exponents $\frac{r'}{s} =\frac{1}{s-q}$ and $\left(\frac{r'}{s}\right)'= \frac{1}{1+q-s}$, and Fubini's theorem, we estimate 
		\begin{equation*}
			\begin{aligned} 
				\int_\Omega (\G \sigma)^{s} d \sigma & \le  \int_{\Omega} u^{q}  
				\left[\G(u^{-\frac{q}{r-1}} d \sigma)\right]^{s-q} 
				 d \sigma \\& 
				 \le \left[\int_{\Omega} \G(u^{-\frac{q}{r-1}} d \sigma)  u^q 
				 d \sigma\right]^{s-q}  \left[ \int_{\Omega}  u^{q} d \sigma\right]^{1+q-s}
				 \\
				& = \left[ \int_{\Omega} \G^{*}(u^q d \sigma) u^{-\frac{q}{r-1}} d \sigma\right]^{s-q} \left[ \int_{\Omega}  u^{q} d \sigma\right]^{1+q-s}\\& 
				\le a^{s-q} \left[ \int_{\Omega} u^{1-\frac{q}{r-1}} d \sigma\right]^{s-q} \left[ \int_{\Omega}  u^{q} d \sigma\right]^{1+q-s}.
			\end{aligned}
		\end{equation*}
		Here $1-\frac{q}{r-1} = \frac{s-(q+q^2)}{s-q}$.
		Setting $s=\frac{q}{1-q}$ where $q+q^2\le 1$, so that $r=\frac{1}{1-q}$,  $r'=\frac{1}{q}\ge s$ and $1-\frac{q}{r-1}=q$, we obtain 
			\[
				\int_\Omega (\G \sigma)^{\frac{q}{1-q}} d \sigma \le a^{\frac{q^2}{1-q}}\int_{\Omega}  u^{q} d \sigma,  
			\]
			for all $0<q\le q_0$.
	  
		If $\sigma$ is a finite measure, $q_0<q<1$,  $s= 1+q$, and $r= \frac{1}{q}+1$, using the preceding estimates we deduce  
			\[
				\int_\Omega (\G \sigma)^{1+q} d \sigma \le a \, \int_\Omega u^{1-q^2} d \sigma 
				\le a \, \left[\int_\Omega u^{q} d \sigma \right]^{\frac{1-q^2}{q}} \left[\sigma(\Omega)\right]^{\frac{q+q^2-1}{q}}.   
			\]
		Hence, for $0<s\le 1+q$, by Jensen's inequality,
		\begin{equation*}
			\begin{aligned} 
				\int_\Omega (\G \sigma)^{s} d \sigma
					& \le \left[ \int_\Omega (\G \sigma)^{1+q} d \sigma\right]^{\frac{s}{1+q}} \left[\sigma(\Omega)\right]^{1- \frac{s}{1+q}} \\ 
					& \le a^{\frac{s}{1+q}} \left[\int_\Omega u^{q} d \sigma\right]^{\frac{s(1-q)}{q}} \left[\sigma(\Omega)\right]^{1-\frac{s(1-q)}{q}}.   
			\end{aligned}
		\end{equation*}
	\end{proof}

	\begin{remark}
		Inequality \eqref{energy} with $s=\frac{q}{1-q}$ is known for quasimetric kernels provided a supersolution $u$ satisfying \eqref{int-sup} exists.
	\end{remark}

\subsection{Construction of supersolutions}
	In the following, we construct a supersolution $\phi\in L^1(\Omega, \sigma)$ to the problem
		\[ \phi \ge [\mathbf{G}(\phi d \sigma)]^q > 0 \quad d \sigma  \, \text{-a.e.} \, \, \text{in} \, \, \Omega. \]
	
	\begin{remark}
		\label{phi_rmk}
		If $\phi$ solves the above inequality, then $u=\phi^{\frac{1}{q}}$ solves \eqref{int-sup}.
	\end{remark}
	
	We then are able to use the energy estimates shown above to construct positive solutions to the integral equation \eqref{int-eq} when the kernel $G$ is non-degenerate.

	The existence of supersolutions will follow from a lemma due to Gagliardo \cite{G} (see also Szeptycki \cite{Szeptycki}). Let $B$ be a Banach space. 
    A convex cone  $P\subset B$ is \textit{strictly convex at the origin} if the convex combination of two elements of $P$ equals zero only if both of those elements are zero, i.e., $\alpha \phi_1 + \beta \phi_2 = 0$  implies $\phi_1 = \phi_2 = 0$, whenever $\alpha, \beta > 0$ and $\alpha + \beta = 1$.

    \begin{lemma}[Gagliardo \cite{G}]
    	\label{gag_lemma}
        Let $B$ be a Banach space and let $P \subset B$ be a convex cone which is strictly convex at the origin.  Let $S\colon \,  P \rightarrow P$ be a continuous mapping.  Assume the following conditions hold:
        \begin{enumerate}
        	\item If $ (\phi_n) \subset P$, $\phi_{n+1} - \phi_n \in P$, and if $\Vert \phi_n \Vert_B \le M$ for all $n = 1, 2, \dots$, then there exists $\phi \in P$ such that $\Vert \phi_n - \phi \Vert_B \rightarrow 0$.
        	\item For $\phi, \psi \in P$, such that $\phi - \psi \in P$, we have $S\phi - S\psi \in P$.
        	\item If $\Vert \phi \Vert_B \le 1$ and if $\phi \in P$, then $\Vert Su \Vert_B \le 1$.
        \end{enumerate}
        	Then for every $\lambda > 0$ there exists $\phi \in P$ such that $( 1 + \lambda)\phi - S\phi \in P$ and $0 < \Vert \phi \Vert_B \le 1$. Moreover, for every $\psi \in P$ such that $0<\Vert \psi \Vert_B \le \frac{\lambda}{1+\lambda}$, $\phi$ can 
	be picked so that $\phi =\psi + \frac{1}{1+\lambda} S \phi$. 
    \end{lemma}
    
    We will apply this lemma to $B= L^1(\sigma)$ and $P \defeq \{ \phi \in L^1(\sigma)\colon \, \phi \ge 0 \, d\sae \}$.
    In our case, it is easy to see that  Lemma \ref{gag_lemma} gives not only that $\Vert \phi \Vert_B > 0$, but further that  $\phi > 0$ $d \sigma$-a.e.
    
	\begin{lemma}
		\label{phi_soln}
		Let $(\Omega, \sigma)$ be a sigma-finite measure space.
		Suppose the strong type inequality \eqref{main_strong_inequality} holds.  
		Then, for every $\lambda > 0$, there is a positive supersolution $\phi \in L^1(\sigma)$ such that 
			\[ \phi \ge [\mathbf{G}(\phi d \sigma)]^q \]
			with $\Vert \phi \Vert_{L^1(\sigma)} \le (1+\lambda)^{\frac{1}{1-q}} \varkappa^{\frac{q}{1-q}}$.
	\end{lemma}
	\begin{proof}
		The supersolution $\phi_0$ can be constructed using Lemma \ref{gag_lemma}.
		Indeed, let $S\colon \,  L^1(\sigma) \rightarrow L^1(\sigma)$ be given by \[S \phi \defeq \frac{1}{\varkappa^q} [\mathbf{G}(\phi d \sigma)]^q\]
		for all $\phi \in L^1(\sigma)$.
		Inequality \eqref{main_strong_inequality} gives that $S$ is a continuous operator.
		Moreover, by \eqref{main_strong_inequality} we can establish condition (3) of Lemma \ref{gag_lemma}.
		Suppose that $\Vert \phi \Vert_{L^1(\sigma)} \le 1$, then
		\begin{align*}
			\Vert S(\phi) \Vert_{L^1(\sigma)}
				&= \frac{1}{\varkappa^q} \int_\Omega [\mathbf{G}(\phi d \sigma)]^q \, d\sigma \\
				&\le \frac{1}{\varkappa^q} \varkappa^q \left(\int_\Omega \phi \, d\sigma\right)^q \\
				&\le 1.
		\end{align*}
		Therefore, by Lemma~\ref{gag_lemma}, there exists $\phi \in L^1(\sigma)$ such that 
			\[ (1+\lambda) \phi \ge \frac{1}{\varkappa^q} [\mathbf{G}(\phi d \sigma)]^q, \]
			$\Vert \phi \Vert_{L^1(\sigma)} \le 1$, and $\phi > 0$ $d \sigma$-a.e.
		We can renormalize with $\phi_0 \defeq a \phi$, with the choice
			\[ a \defeq \left[ \frac{1}{(1 + \lambda) \varkappa^q} \right]^{\frac{1}{1-q}}, \]
			and see that $\phi_0$ satisfies
			\[ \phi_0 \ge [\mathbf{G}(\phi_0 d \sigma)]^q, \]
			with 
			\[\Vert \phi_0 \Vert_{L^1(\sigma)} \le (1 + \lambda)^{\frac{1}{1-q}} \varkappa^{\frac{q}{1-q}},\] and $\phi_0 > 0$ $d \sigma$-a.e.
	\end{proof}

	\begin{lemma}\label{supsoln_to_soln} If there exists a positive supersolution $u_0\in L^q(\Omega, \sigma)$ satisfying \eqref{int-sup}, then there exists a positive solution $v \in L^q(\Omega, d \sigma)$ such that $v = \mathbf{G}(v^q d \sigma)$ $d \sigma-{\rm a.e.}$, unless $G$ is degenerate.
		In the latter case of the degenerate kernel, the equation 
		$v = \mathbf{G}(v^q d \sigma)$ does not have a positive solution $v \in L^q(\Omega, \sigma)$. 
	\end{lemma}
	\begin{proof}
		Let $u_0\in L^q(\Omega, \sigma)$ be the positive supersolution to \eqref{int-sup}.
		We can define by induction  the non-increasing sequence of  supersolutions $\{ u_n \}_{n=0}^{\infty}$ given by
			\[ u_{n+1} \defeq \textbf{G}[u_n^q d \sigma], \quad n=0, 1, 2, \ldots,\] 
		where $u_n\downarrow v$, and  $v\in L^q(\Omega, d \sigma)$ is a nonnegative solution by the Dominated Convergence Theorem. 
		
		It remains to check that the solution $v$ is positive $d \sigma$-a.e. provided the kernel is non-degenerate. 
	 	This can be done by finding a  lower bound on the supersolutions $u_n$ by using Lemma~\ref{lemma2} with $u_n$ in place of $u$ and $\sigma_K$ in place of $\sigma$, for an arbitrary compact set $K\subset \Omega$. Notice that by induction each $u_n>0$ $d \sigma$-a.e. since $G$ is non-degenerate. Consequently,   
	 		\[\int_K (\mathbf{G}\sigma_K)^s d \sigma  \le C_K \left[\int_K u_n^q \, d \sigma \right]^r,\]
	 	where $s=\min \left(\frac{q}{1-q}, 1+q\right)$,  $r>0$, and $C_K$ does not depend on $n$. Letting $n \to +\infty$ in the preceding inequality, we deduce 
		\[\int_K ( \mathbf{G}\sigma_K)^s d \sigma  \le C_K \left[\int_K v^q \, d \sigma \right]^r.\]
Thus, if $v=0$ on $K$ then $\mathbf{G}\sigma_K=0$ $d \sigma$-a.e. on $K$, and hence $G(\cdot, y)=0$  $d \sigma$-a.e. for $y\in K$. 
	 	Hence,  $\sigma(K)=0$, i.e., $v>0$ $d \sigma$-a.e.
	 	
		If the kernel is degenerate, then clearly a positive solution does not exist. Indeed, 
	 		 	if $G$ were degenerate, then there exists a set $K$ such that $\sigma(K) > 0$ and $G(x, \cdot)=0$  $d \sigma$-a.e. for $x \in K$.
	 	This implies that, for every solution $u(x) = \int_\Omega G(x,y) u^q \, d\sigma(y) = 0$ for $x \in K$ $d \sigma$-a.e.,which shows that  
	 	 a positive solution $u$ does not exist. 
	\end{proof} 
	
	\begin{cor}
		If inequality \eqref{main_strong_inequality} holds, and there exists a solution $u \in L^q(\Omega, \sigma)$ to \eqref{int-eq}, it follows that 
			\[ \Vert u \Vert_{L^q(\Omega, \sigma)} \le \varkappa^{{\frac{1}{1-q}}}. \]
	\end{cor}
	\begin{proof}
		By applying  \eqref{main_strong_inequality}  to $\nu \defeq u^q \sigma$, we get
		\[ \nu(\Omega) = \int_\Omega (\mathbf{G}\nu)^q \, d\sigma \le \varkappa^q \nu(\Omega)^q.\]
	\end{proof}

\subsection{Derivation of inequality}
	In establishing a converse result, we appeal to Potential Theory, and in particular some results due to Fuglede \cite{F}.
	The necessary definitions and results are summarized in Section \ref{capacity_theory} above.

	We will need the following weak-type inequality.
	\begin{lemma}\label{weak_frac_bound}
		Let $G$ be a symmetric, nonnegative kernel satisfying a weak maximum principle. 
		Suppose $\omega \in \mathcal{M}^+(\Omega)$ is absolutely continuous with respect to capacity.
		Then
			\begin{equation}\label{weak-quotient}
			 \left\Vert \frac{\mathbf{G}\nu}{\mathbf{G}\omega} \right\Vert_{L^{1,\infty}(\Omega, \omega)} \le h \Vert \nu \Vert,
			 \end{equation} 
		for any $\nu \in \mathcal{M}^+(\Omega)$.
	\end{lemma}
	\begin{proof}
		Let $t > 0$.
		Define $E_t \defeq \{ x \in \Omega\colon \frac{\mathbf{G}\nu}{\mathbf{G}\omega} 
		(x)  > t \}$.
		We claim that compact subsets $K \subset E_t$ have finite capacity.
		This requires that $G(x,x) > 0$ on $E_t$.
		Letting $A \defeq \{ x \in \Omega : G(x,x) = 0 \}$, we claim $A \cap E_t = \emptyset$.
		Indeed, by the weak maximum principle, since $\mathbf{G} \delta_x (x) = 0$ for any $x \in A$, then we have $\mathbf{G} \delta_x (y) = 0$ for every $y \in \Omega$.  
		Thus, $G(x,y) = 0$ on $A \times \Omega$.  
		Further, for any measure $\nu \in \mathcal{M}^+(\Omega)$, $\mathbf{G}\nu (x) = 0$ for $x \in A$.
		Adapting the convention $\frac{0}{0} = 0$, we see then that $E_t \cap A = \emptyset$ as claimed.
		
		Let $K \subset \Omega$ be a compact set.
		We can find an equilibrium measure $\mu \in \mathcal{M}^+(K)$ such that
			$\mathbf{G}\mu \ge 1$ n.e. on $K$ and $\mathbf{G}\mu \le 1$ on $S_\mu$. 
		Thus, if $N \defeq \{ x \in K : \mathbf{G}\mu(x) < 1 \}$, then we have $\omega(N) = 0$, since $\omega$ is absolutely continuous with respect to capacity. 
		By the weak maximum principle,  $\mathbf{G}\mu \le 1$ on $S_\mu$ 
		yields $\mathbf{G}\mu \le h$ on $\Omega$.
		
		We deduce the following estimate
			\begin{align*}
				\omega(K) & \le \int_K \G\mu \, d\omega = \int_K \mathbf{G}\omega_K \, d\mu \\
					&\le \int_K \frac{\mathbf{G}\nu}{t} \, d\mu = \frac{1}{t} \int_\Omega \mathbf{G}\mu \, d\nu \\  &\le \frac{1}{t} \int_\Omega h \, d\nu = \frac{h}{t} \nu(\Omega). 
			\end{align*}
		Therefore we have $\omega(K) \le \frac{h \nu(\Omega)}{t}$ for any compact set $K \subset E_t$.
		Taking the supremum over all such $K$, we find	
			\[ \omega(E_t) \le \frac{h}{t} \nu (\Omega),\]
		for all $t>0$. This establishes \eqref{weak-quotient}. 
	\end{proof}

	\begin{lemma}\label{supsoln_impl_ineq}
		Let $G$ be a quasi-symmetric kernel which satisfies the weak maximum principle.
		Suppose there is a positive supersolution $u$ to \eqref{int-eq} such that $u \in L^q(\Omega, \sigma)$. 
		Then  \eqref{main_strong_inequality} holds.
	\end{lemma}

	\begin{proof} Without loss of generality we may assume that $G$ is symmetric 
		 (see Remark~\ref{symmetrized_kernel}). 
		Let $u \in L^q(\Omega, \sigma)$ be a positive supersolution, i.e., $\mathbf{G}(u^q \sigma) \le u$.
		Let the measure $\omega$ be given by $d \omega \defeq u^q \, d\sigma$.
	By Lemma \ref{soln_abs_cont}, we know that $\omega$ is absolutely continuous with respect to capacity. 	Suppose $\nu \in \mathcal{M}^+(\Omega)$.
		If $\nu(\Omega) = + \infty$, there is nothing to prove.
		In the case that $\nu(\Omega) < + \infty$, we can normalize the measure and work with the case that $\nu(\Omega) = 1$.
		
		Since $u$ is a positive supersolution, we have $(\mathbf{G}\omega)^q d \sigma \le d \omega$. We estimate 
		\begin{align*}
			\int_\Omega (\mathbf{G}\nu)^q \, d\sigma &= \int_\Omega \left( \frac{\mathbf{G}\nu}{u} \right)^q u^q \, d\sigma \\
				&\le \int_\Omega \left( \frac{\mathbf{G} \nu}{\mathbf{G} \omega} \right)^q \, d\omega \\
				&= q \int_0^\beta \omega\left( \Big\{\frac{\mathbf{G}\nu}{\mathbf{G}\omega} > t \Big\}\right) t^{q-1} \, dt + q \int_\beta^\infty \omega\left(\Big\{\frac{\mathbf{G}\nu}{\mathbf{G}\omega} > t \Big\}\right) t^{q-1} \, dt \\
				&= I + II,
		\end{align*}
		for any $\beta>0$. 
		
		For integral $I$, we see that $I \le \beta^q \omega(\Omega) = \beta^q \int_\Omega u^q \, d\sigma$.

		By Lemma \ref{weak_frac_bound}, we have the weak type bound
			\[ \omega\left(\Big\{\frac{\mathbf{G}\nu}{\mathbf{G}\omega} > t\Big\} \right) \le \frac{h \nu(\Omega)}{t} = \frac{h}{t}. \]
		With this estimate, we find $II \le \frac{q}{1-q} h \beta^{q-1}$.
		Thus, with a choice of $\beta = \frac{h}{\omega(\Omega)}$, we deduce 
			\[ \int_\Omega (\mathbf{G}\nu)^q \, d\sigma \le \frac{h^q }{1-q} \left( \int_\Omega u^q \, d\sigma \right)^{1-q}. \]
		Therefore, in the general case with $\nu \in \mathcal{M}^{+}(\Omega)$, we obtain the desired inequality
			\[ \int_\Omega (\mathbf{G}\nu)^q \, d\sigma \le \frac{h^q }{1-q} \left( \int_\Omega u^q \, d\sigma \right)^{1-q} \nu(\Omega)^q. \]
		It is important to note that in the above inequality, we have the constant on the right hand side in terms of the norm $\Vert u \Vert_{L^q(\Omega, \sigma)}$.
		This implies that \eqref{main_strong_inequality} holds with  
			\[ \varkappa \le  \frac{h}{(1-q)^{\frac{1}{q}}}  \, \Vert u \Vert_{L^q(\Omega, \sigma)}^{1-q}, \]
		where $\varkappa$ is the least constant in \eqref{main_strong_inequality}.
	\end{proof}

\section{Weak type results}\label{weak-section}
	
	In addition to characterizing the strong type inequality \eqref{main_strong_inequality}, we study in this section the analogous weak type $(1, q)$-inequality 
	\begin{equation}\label{weak-X-Y}
	\Vert  \mathbf{G}\nu \Vert_{L^{q,\infty}(X, \sigma)} \le C \,   \Vert \nu \Vert \quad \text{\rm for all} \,\,   \nu \in \mathcal{M}^+(Y),
	\end{equation}  
	in a more general setting   where $G$ is a kernel on $X\times Y$ and $\sigma \in \mathcal{M}^+(X)$.  
	We give various characterizations of \eqref{weak-X-Y} 
	using capacities,  as well as non-capacitary terms, for all $0<q<\infty$.

	A complete characterization of \eqref{weak-X-Y}  in terms of the capacity 
	$\text{cap}_0 (\cdot)$ (see Sec.~\ref{capacity_theory} above) is  given in the following proposition. 
	Note that this result does not require $G$ to satisfy the weak maximum principle on  $\Omega$, does not restrict to the case $X=Y$, and does not place any restriction on the range of $q>0$.

	\begin{prop}\label{weak-prop} 
		Let $G$ be a kernel on $X \times Y$. Suppose $0 < q<+\infty$ and $\sigma\in \mathcal{M}^{+}(X)$.  
		Then there exists a positive constant $C$ such that \eqref{weak-X-Y} holds  if and only if
			\begin{align}\label{q-cap}
				\sigma(K) \le C^q \, \left(\capa_0 (K)\right)^q, \quad \text{\textnormal{for all compact sets}}
				 \, \,  K \subset X,  
			\end{align}
		where  $C$ is the same between both statements.
	\end{prop}
		\begin{proof}
	($\Rightarrow$) Without loss of generality we may assume that $C=1$. Let $K \subset X$ be a compact set.
		If $\capa_0 K = + \infty$, there is nothing to show, 
		so we assume $\capa_0 K < +\infty$.
		Then for every $\epsilon > 0$, there exists a measure $\lambda \in \mathcal{M}^+(Y)$ so that $\mathbf{G}\lambda(x) \ge 1$ on $K$ and $\lambda(Y) \le \capa_0(K) + \epsilon$.
		Then by \eqref{weak-X-Y}, 
		\begin{align*}
			\sigma(K) &\le \Vert\mathbf{G}\lambda\Vert^q_{L^{q, \infty}(K, \sigma)} \\
				&\le  \lambda(Y)^q \\
				&\le   \left( \capa_0(K) + \epsilon \right)^q. 
		\end{align*}
		Letting $\epsilon \rightarrow 0$, we establish the capacity inequality \eqref{q-cap}.
		
	($\Leftarrow$) Suppose $\sigma(K) \le \left(\capa_0 (K)\right)^q$ for any compact $K \subset X$.
		For $t>0$, let $E_t \defeq \{ x \in X : \mathbf{G}\nu(x) > t \}$.
		Let $K \subset \Omega_t$ be a compact set.
		For $\epsilon > 0$, 
		by the dual definition of capacity \eqref{cont-def}, we can find a measure $\mu \in \mplus{K}$ such that $\mathbf{G}^*\mu (y) \le 1$ for all $y \in Y$ and $\capa_0 (K) \le \mu(K) + \epsilon$.

		Then by Fubini's theorem, 
		\begin{align*}
			\sigma(K)^{\frac{1}{q}}
				&\le \mu(K) + \epsilon \\
				&\le \frac{1}{t} \int_K \mathbf{G}\nu(x) \, d\mu(x) + \epsilon \\
				&= \frac{1}{t} \int_Y \mathbf{G}^*\mu(y) \, d\nu(y) + \epsilon\\
				&\le \frac{ \nu(Y)}{t} + \epsilon. 
		\end{align*}
		By exhausting over all compact sets $K \subset E_t$ and letting $\epsilon \rightarrow 0$, we establish the weak type $(1, q)$ inequality
		\begin{equation*}
		\sigma(E_t)^{\frac{1}{q}} \le \frac{ \nu(Y)}{t}, 
		\end{equation*}
		for all $t>0$, which proves \eqref{weak-X-Y}. 
	\end{proof}

	In the case $q > 1$ we can use the duality 
	$L^{q, \infty}(X, \sigma) = [L^{q', 1}(X, \sigma)]^*$,  $\frac{1}{q}+\frac{1}{q'} = 1$, to show  
 that it suffices to verify \eqref{weak-X-Y} on  point masses $\nu = \delta_x$, 
	$x \in X$. This leads to a simple non-capacitary characterization of \eqref{weak-X-Y}. 

	\begin{prop}
		\label{imbeddinglargeq}
		Let $G$ be a kernel on $X \times Y$.
		Suppose $1 < q < +\infty$, and $\sigma \in \mathcal{M}^{+}(X)$.
		Then the following statements are equivalent.
		\begin{enumerate}
			\item There exists a positive constant $C$ such that \eqref{weak-X-Y} holds. 
				
			\item The following condition holds,
				\begin{equation}\label{statement 2}
					\sup_{y \in Y}  \Vert G(\cdot, y) \Vert_{L^{q, \infty}(X, \sigma)}  < +\infty. 
				\end{equation}
				
			\item There exists a positive constant $C$ such that, for all measurable sets $E \subset X$, 
				\begin{equation}\label{statement 3}
					\sup_{y \in Y} \,  \mathbf{G}^*\sigma_E (y) \le C \,  \sigma(E)^{\frac{1}{q'}}.		
				\end{equation}
					\end{enumerate}
	\end{prop}
	
	\begin{proof}
		By duality,  statement (1) is equivalent to 
			\[ \int_X (\mathbf{G} \nu) \phi \, d\sigma \le c \Vert \nu \Vert \Vert \phi \Vert_{L^{q', 1}(\sigma)}, \quad \forall \phi \in L^{q',1}(X, \sigma), \nu \in \mathcal{M}^{+}(Y). \]
		Equivalently, by Fubini's theorem, 
			\[ \int_Y \mathbf{G^*}(\phi \sigma) \, d\nu \le c \Vert \nu \Vert \Vert \phi \Vert_{L^{q', 1}(\sigma)}, \quad \forall \phi \in L^{q',1}(X, \sigma), \nu \in \mathcal{M}^{+}(Y). \]
		Clearly, the preceding inequality holds if and only it holds for 
		all $\nu = \delta_y$,  that is, 
			\[    \mathbf{G^*}(\phi \sigma)(y)=  \int_X G(x,y) \phi(x) \, d\sigma(x) \le c \Vert \phi \Vert_{L^{q', 1}(\sigma)} , \quad \forall \phi \in L^{q',1}(X, \sigma), y \in Y. \] 
	Using duality again, we see that the preceding inequality is equivalent to 
		\begin{equation}\label{eq-y} 
		\Vert G(\cdot, y) \Vert_{L^{q,\infty}(\sigma)} \le c, \quad \forall  y \in Y. 
		\end{equation}
		This establishes (1)$\Longleftrightarrow$(2).
		
		The equivalence  (2)$\Longleftrightarrow$(3) follows from 
		the well-known fact that, for $q>1$,   $\Vert f\Vert_{L^{q, \infty}(X, \sigma)}$   is equivalent to the norm 
					\[ \sup_{E\subset X} \frac{1}{\sigma(E)^{\frac{1}{q'}}} \int_{E} |f | \, d\sigma(x). 
		 \]
				Applying this to $f(\cdot) = G (y, \cdot)$, for a fixed $y \in Y$, we see that 
				\[
				 \mathbf{G}^*\sigma_E (y)=\int_E G(y, x) \, d \sigma(x) \le C \, \sigma(E)^{\frac{1}{q'}},
				\] 
				where $C$ does not depend on $y\in Y$ and $E\subset X$, if and only if \eqref{eq-y} holds. 
				\end{proof}
	
	\begin{remark}
	For Riesz kernels $I_\alpha(x)=|x|^{\alpha-n}$ $(0<\alpha<n)$ on $\R^n$, condition \eqref{statement 2} means that 
	$\sigma(B(x, r)) \le C \, r^{(n-\alpha)q}$ for all balls $B(x, r)$ in $\R^n$. This condition  
	was used by D. Adams  in the context of $(p, q)$ inequalities for $q>p>1$; the capacitary condition \eqref{q-cap}  
was introduced by 
	V. Maz'ya (see \cite{AH}, \cite{Maz}). 
	\end{remark}

	There are  more direct characterizations of the weak-type 
	$(1, q)$-inequality in the case $0<q \le 1$ if $X=Y=\Omega$, and additionally if $G$ is quasi-symmetric and satisfies the weak maximum principle. Notice that in this case 
	$\text{cap}_0 (\cdot)$ is equivalent to the Wiener capacity $\text{cap}_1 (\cdot)$. 

	\begin{theorem}\label{weak-thm}
		Let $\sigma \in \mathcal{M}^{+}(\Omega)$, and $0<q <\infty$.  
		Suppose $G$ is a quasi-symmetric kernel on $\Omega\times \Omega$ which satisfies the weak maximum principle.
		Then the following statements are equivalent:
		\begin{enumerate}
			\item There exists a positive constant $c$ such that 
			$$\Vert  \mathbf{G}\nu \Vert_{L^{q,\infty}(\Omega, \sigma)} \le c \,   \Vert \nu \Vert \quad \text{\rm for all} \,\,   \nu \in \mathcal{M}^+(\Omega).$$ 
			\item There exists a positive constant $C$ such that 
			$$\sigma(K) \le C \, \left(\capa_1 (K)\right)^q \quad \text{\rm for all compact sets} \,\, K \subset \Omega.$$ 
						\item $\mathbf{G} \sigma \in L^{\frac{q}{1-q}, \infty}(\Omega, \sigma)$, when $0 < q < 1$.
		\end{enumerate}
	\end{theorem}
	The details of this theorem can be found in \cite{QV-Wheeden}.

	We finally consider \eqref{weak-X-Y} in the case $q=1$, i.e. the weak-type $(1, 1)$-inequality, along with its $(p, p)$-analogues for $1<p<+\infty$, under the same assumptions as in Theorem \ref{weak-thm}.  
	
	\begin{theorem}\label{1-1-weak-thm}
		Let $\sigma \in \mathcal{M}^{+}(\Omega)$.  
		Suppose $G$ is a quasi-symmetric kernel on $\Omega\times \Omega$ which satisfies the weak maximum principle.
		Then the following statements are equivalent. 
		\begin{enumerate}
			\item There exists a positive constant $c$ such that 
				\begin{equation}\label{cond-1-1}
				\Vert  \mathbf{G}\nu \Vert_{L^{1,\infty}(\sigma)} \le c \,   \Vert \nu \Vert \quad \text{\rm for all} \,\,   \nu \in \mathcal{M}^+(\Omega).
				 \end{equation} 
			\item If $1<p<+\infty$, then there exists a positive constant $c$ such that 
			\begin{equation}\label{cond-p-p}
			\Vert  \mathbf{G} (f \, d \sigma) \Vert_{L^{p}(\sigma)} \le c \,  	\Vert f \Vert_{L^{p}(\sigma)} \quad \text{\rm for all} \,\,   f \in L^{p}(\Omega, \sigma).
			\end{equation} 
						\item There exists a positive constant $c$ such that 
							\begin{equation}\label{cond-K-K}\iint_{K \times K} G(x, y) \, d \sigma(x) \, d \sigma(y) \le c \, \sigma(K), \quad \text{\rm for all compact sets} \,\, K \subset \Omega.\end{equation} 
					\item 		If $G$ is a quasi-metric kernel then 
											\begin{equation}\label{cond-B-B}\iint_{B \times B} G(x, y) \, d \sigma(x) \, d \sigma(y) \le c \, \sigma(B),\end{equation} 
						\textnormal{for all  quasi-metric balls $B=B(x, r)$, where} $B(x, r)=\{y\in \Omega \colon \, d(x, y)<r\}$, 
						$d(x, y)= \frac{1}{G(x, y)}$ $(x, y \in \Omega, r>0)$. 
\end{enumerate}
	\end{theorem}

\begin{remark}\label{nazarov} The equivalence of statements (2) and (4) of Theorem \ref{1-1-weak-thm} 
in the case of quasi-metric kernels $G$ is due to F. Nazarov (see \cite{NV}, Theorem 4.6); it can be deduced from more general 
results on operators with non-positive kernels in the framework of non-homogeneous harmonic analysis 
(see T. Hyt\"{o}nen \cite{Hy}). The  $(1, 1)$ weak-type  inequality 
in Theorem \ref{1-1-weak-thm} may be new. 
	\end{remark}
	
	\begin{proof} As above in the case of strong-type $(1,q)$ inequalities,  we may assume without loss of generality that $G$ is a symmetric kernel such 
	that $G(x, x)>0$ for all $x \in \Omega$. The latter condition ensures that $\capa_1 (K)<\infty$ for 
	any compact set $K \in \Omega$. 
	By Proposition \ref{weak-prop}, the $(1,1)$ weak-type inequality \eqref{cond-1-1} 
	is equivalent to the condition 
	\begin{equation} \label{cap-1} 
					\sigma(K) \le C \,  \capa_1 (K),	 \quad  \textnormal{for all compact sets} 
				 \, \,  K \subset \Omega. 
				\end{equation} From the discussion in Sec. \ref{capacity_theory} it follows that, for any compact set $K\subset \Omega$,  
$$
 \capa_1 (K) = \sup \Big\{ \mu(K)\colon \,\,  \frac{1}{\mu(K)} \iint_{K\times K} G(x, y)\, d \mu(x) \, d \mu(y) \le 1
 \Big\}, 
$$
where the supremum is taken over all $\mu \in \mathcal{M}^{+}(K)$ such that $\mu(K)>0$. Taking 
$\mu = \frac{1}{C}  \, \sigma$, where $C$ is the constant in \eqref{cap-1}, we see that 
\eqref{cond-K-K} implies \eqref{cap-1}, and consequently, \eqref{cond-1-1}. This proves  (3)$\Longrightarrow$(1). 

Conversely, suppose that \eqref{cap-1} holds. Let $1<p<+\infty$. We first  prove the corresponding 
weak-type $(p, p)$ inequality
\begin{equation}\label{weak-p-p}
			\Vert  \mathbf{G} (g \, d \sigma) \Vert_{L^{p, \infty}(\sigma)} \le c \,  	\Vert g \Vert_{L^{p}(\sigma)},
			\end{equation} 
		where $c$ is independent of $g$. 	Here without loss of generality we may assume that  $g \in L^p(\Omega, \sigma)$,  $g \ge 0$ is compactly supported. 
		For a fixed $t>0$, denote by $E_t$ the 
			set $$E_t=\{ x \in \Omega\colon \, \, 
			\mathbf{G} (g \, d \sigma)(x)>t\}.$$ Notice that 
			$$\mathbf{G} (g \, d \sigma_{E_t^c})\le \mathbf{G} (g \, d \sigma)\le t \quad \text{on} \, \, E_t^c.$$
		Consequently, by the weak maximum principle 
			$$
			\mathbf{G} (g \, d \sigma_{E_t^c}) \le  h \, t \quad \text{on} \, \, \Omega.
			$$
			Denote by $K$ an arbitrary compact subset of  the  set $F_t$  defined by 
			$$F_t=\{x \in \Omega \colon \, \, \mathbf{G} (g \, d \sigma)(x)> (h+1) \, t\}.$$
		We observe that by the preceding estimates,
		$$F_t \subset \Big \{x \in \Omega \colon \, \, \mathbf{G} (g \, d \sigma_{E_t}) >t \Big\}. $$
			We denote by $\mu$  the equilibrium measure $\mu$ associated with 
			$ \capa_1 (K)$, which is supported on $K$, and has the property $\mathbf{G} \mu \le 1$ on $K$.
			 Hence 
$\mathbf{G} \mu \le h$ on $\Omega$ by the weak maximum principle.

			Since $K \subset F_t$,  by \eqref{cap-1}  
		 we estimate 
			\begin{align*}
			\sigma(K)\le C \, \capa_1 (K) & = 
			C \, \mu(K)\\ & \le \frac{C}{t} \, \int_{K} \mathbf{G} (g \, d \sigma_{E_t}) \, d \mu \\ 
			&= \frac{C}{t} \int_{\Omega} (\mathbf{G} \mu) \, g \, d \sigma_{E_t} \\ &\le \frac{C\, h}{t} \int_{E_t}  g \, d \sigma. 
			\end{align*}
			 From this by Jensen's inequality we deduce 
$$
\sigma(K)\le  \frac{C \, h}{t} \, \sigma(E_t)^{\frac{1}{p'}} \Vert g\Vert_{L^{p}(\sigma)}. 
$$
Taking the supremum over all $K \subset F_t$, we see that 
	$$
\sigma(F_t) \le \frac{C \, h}{t} \,  \sigma(E_t)^{\frac{1}{p'}} \Vert g\Vert_{L^{p}(\sigma)}. 
$$
Multiplying both sides of the preceding inequality by $t^p$ and taking the supremum 
over all $t \in (0, t_0)$ we obtain 
\begin{align*}
\sup_{0<t<t_0} \, \Big[ t^p \, \sigma(F_t)\Big] \le C \, h \, \sup_{0<t<t_0} \, \Big[ t^p \, \sigma(E_t)
\Big]^{\frac{1}{p'}} \Vert g\Vert_{L^{p}(\sigma)}.
\end{align*}
Here the right-hand side is finite for any $t_0>0$ since $g$ is compactly supported, and 
consequently $g \in L^{1}(\Omega, \sigma)$, so that 
$$
\sup_{0<t<t_0} \,  \Big[t^p \, \sigma(E_t)\Big] \le t_0^{p-1} \, \sup_{0<t<\infty} \,  \Big[ t \, \sigma(E_t) \Big]
\le  t_0^{p-1} \, \Vert g\Vert_{L^{1}(\sigma)}<\infty. 
$$
Notice that 
\begin{align*}
\sup_{0<t<t_0} \, \Big[ t^p \, \sigma(F_t)\Big] & = \frac{1}{(h+1)^p} \, \sup_{0<\tau<(h+1)t_0} \,  \Big[   \tau^p \, \sigma(E_\tau)\Big] \\ 
& \ge \frac{1}{(h+1)^p} \, \sup_{0<\tau<t_0} \, \Big[ \tau^p \, \sigma(E_\tau)\Big].
\end{align*}
Combining the preceding estimates we deduce 
\begin{align*}
\sup_{0<\tau<t_0} \, \Big[\tau^p \, \sigma(E_\tau) \Big]^{\frac{1}{p}} 
\le C \, h \, (h+1)^p \,  \Vert g\Vert_{L^{p}(\sigma)}.
\end{align*}
Letting $t_0\to +\infty$, we obtain 
\begin{align*}
\sup_{0<\tau<+\infty} \, \Big[\tau^p \,  \sigma(E_\tau) \Big]^{\frac{1}{p}} 
\le C \, h \, (h+1)^p \,  \Vert g\Vert_{L^{p}(\sigma)}.
\end{align*}
This proves the weak-type $(p, p)$-inequality \eqref{weak-p-p} for all $1<p<+\infty$, which  
by the Marcinkiewicz interpolation theorem yields \eqref{cond-p-p} for all $1<p<+\infty$. 

For any measurable set $
E \subset \Omega$ and 
$1<p<+\infty$, letting $g=\chi_E$ in \eqref{cond-p-p} (or \eqref{weak-p-p}),  we deduce by Jensen's inequality
\begin{equation}\label{cond-E-E}
\iint_{E \times E} G(x, y) \, d \sigma(x) \, d \sigma(y) \le \Vert \mathbf{G} (\chi_E \, \sigma)\Vert_{L^p(\sigma)} 
\, \sigma(E)^{\frac{1}{p'}} \le  C \, \sigma(E). 
\end{equation}
In particular, \eqref{cond-K-K} and  \eqref{cond-B-B} hold. Thus proves (1)$\Longrightarrow$(2)$\Longrightarrow$(3)$\Longrightarrow$(1). 

If $G$  is a quasi-metric kernel, then (4)$\Longrightarrow$(2) for $p=2$; see Remark \ref{nazarov}. Conversely, 
\eqref{cond-p-p} for $p=2$ yields \eqref{cond-E-E} for any measurable $E\subset\Omega$, so that (2)$\Longrightarrow$(4). 
\end{proof}

\section{Breaking the inequality: a counterexample}
	\label{broken_inequality}

	In this section, we provide some examples which demonstrate that our main results may fail in the absence 
	of the weak maximum principle, first for non-negative symmetric kernels $G$, and then for strictly positive 
	kernels. More specifically, we justify the following remarks.

	\begin{remark}\label{remark7.1} 
	Without the weak maximum principle, for a symmetric kernel $G$ there can be a positive solution to $u = \mathbf{G} (u^q d \sigma)$ with $u \in L^q(\Omega, \sigma)$ but there is no constant $0 < \varkappa < +\infty$ such that the inequality $\int_\Omega (\mathbf{G}\nu)^q \, d \sigma \le \varkappa^q \nu(\Omega)^q$ holds for all $\nu \in \mathcal{M}^+(\Omega)$.
	\end{remark} 
		
	First, we present some minor computations for $2\times 2$ matrices which we will employ extensively below.
	Suppose that we have a discrete kernel $G(x_i, x_j) = g_{i j}$ ($i=1, 2$)  on $\Omega=\{x_1, x_2\}$, 
	 where 
	$x_1, x_2$ are distinct points,  and 
		\[ G = \left[g_{ij}\right] = \left[ \begin{tabular}{cc}0 & 1 \\ 1 & 0 \end{tabular} \right]. \] 
		Note that this kernel does not satisfy the weak maximum principle.

	Suppose we have the measure $\sigma = (\sigma_1, \sigma_2)$ on $\Omega$, and $u = (u_1, u_2)$, where 
	$u_i, \sigma_i\ge 0$ ($i=1, 2$).
	Then, if $u$ is a solution to the equation $u=\mathbf{G} (u^q d\sigma)$, we have the system of equations:
		\[ u_1 = u_2^q \sigma_2, \quad u_2 = u_1^q \sigma_1, \]
	 which we can solve explicitly for $u$ in terms of $q$ and $\sigma$:
			\begin{align*} 
			u_1 = (\sigma_1^q \sigma_2)^\frac{1}{1-q^2}, \quad u_2 = (\sigma_2^q \sigma_1)^\frac{1}{1-q^2}.  
				\end{align*}
	We compute the norm of $u$ in $L^q(\sigma)$ to be 
		\begin{align*}
			\Vert u \Vert^q_{L^q(\sigma)} 
				&= u_1^q \sigma_1 + u_2^q \sigma_2 \\
				&= (\sigma_1^q \sigma_2)^\frac{q}{1-q^2} \sigma_1 + (\sigma_1 \sigma_2^q)^\frac{q}{1-q^2} \sigma_2.
		\end{align*}

	Now suppose we have a kernel $G$ on the discrete set of distinct points $\Omega=\{x_k\}_{k=1}^{\infty}$.  This kernel will consist of the above blocks placed along the diagonal and zero elsewhere:
		\begin{align}\label{b-matrix}
			G =
					\begin{bmatrix}
					0&1 \\
					1&0 \\
					 & &0&1 \\
					 & &1&0 \\
					 & & & &0&1\\
					 & & & &1&0\\
					 & & & & & &\ddots
					\end{bmatrix}  
		\end{align}
	Then we find that for $\sigma = (\sigma_k)_{k=1}^{\infty}$, as above, the equation $u = \mathbf{G}(u^q d \sigma)$ has a solution  
			\begin{align}\label{u-sol} u & = (u_1, u_2, \dots, u_{2k-1}, u_{2k}, \dots) \\ &= ((\sigma^q_{1} \sigma_{2})^\frac{1}{1-q^2}, (\sigma_{1} \sigma_{2}^q)^\frac{1}{1-q^2}, \dots, (\sigma^q_{2k-1} \sigma_{2k})^\frac{1}{1-q^2}, (\sigma_{2k-1} \sigma_{2k}^q)^\frac{1}{1-q^2}, \dots),  \notag
					\end{align}
	with norm
		\begin{align} \label{u-norm}
			\Vert u \Vert_{L^q(\sigma)}^q 
				&= \sum_{k=1}^\infty u_k^q \sigma_k\\
				&= \sum_{k=1}^\infty  \left( (\sigma_{2k-1}^q \sigma_{2k})^\frac{q}{1-q^2} \sigma_{2k-1} + (\sigma_{2k-1} \sigma_{2k}^q)^\frac{q}{1-q^2} \sigma_{2k} \right). \notag
		\end{align}
		
	We would now like to create a measure $\sigma$ for which $\Vert u \Vert_{L^q(\sigma)} < +\infty$.
	Set $\sigma_{2k-1} = a^k$ and $\sigma_{2k} = b^{-k}$.
	Then the $k$-th pair of terms in the sum are
		\begin{align*}
			(\sigma_{2k-1}^q \sigma_{2k})^\frac{q}{1-q^2} \sigma_1 + (\sigma_{2k-1} \sigma_{2k}^q)^\frac{q}{1-q^2} \sigma_{2k}
				&= (a^{kq} b^{-k})^\frac{q}{1-q^2} a^k + (a^k b^{-kq})^\frac{q}{1-q^2} b^{-k} \\
				&= \left[ \left( \frac{a^q}{b} \right)^\frac{q}{1-q^2} a \right]^k + \left[\left( \frac{a}{b^q} \right)^\frac{q}{1-q^2} \frac{1}{b} \right]^k \\
				&= \left[ \left( \frac{a}{b^q} \right)^\frac{1}{1-q^2} \right]^k + \left[\left( \frac{a^q}{b} \right)^\frac{1}{1-q^2} \right]^k. 
		\end{align*}
	We wish to choose $a, b>0$ so that 
		\[ a < b^q, \quad a^q < b. \]
		Note that this reduces down to choosing $1<a < b^q$. 
	If this holds, then $a^q < a < b^q < b$, so $a^q < b$.
	Therefore, with appropriate choices of $a, b$, we have $\Vert u \Vert_{L^q(\sigma)} < +\infty$.

	Now we wish to  show that 
 \begin{equation}\label{inf-sum}
		 \sup_{\nu \in \mathcal{M}^+(\Omega)} \frac{ \int_\Omega (\mathbf{G}\nu)^q \, d\sigma }{ \nu(\Omega)^q} = + \infty.
		 \end{equation}
		 Note that the ratio on the left hand side can be written as
	\begin{equation*}
		\frac{ \int_\Omega (\mathbf{G}\nu)^q \, d\sigma}{\nu(\Omega)^q}
			= \frac{ \sum_{k=1}^\infty  (\nu_{2k}^q \sigma_{2k-1} + \nu_{2k-1}^q \sigma_{2k} ) }{ \left(\sum_{k=1}^\infty \nu_k\right)^q}. 
			\end{equation*}
		Setting $\nu_{2k-1} = \sigma_{2k}^\frac{1}{1-q}$, $\nu_{2k} = \sigma_{2k-1}^\frac{1}{1-q}$ for 
		$k=1, 2, \ldots, n$, and $\nu_k=0$ for $k >2n$, we obtain
			\begin{align*}
			\frac{ \int_\Omega (\mathbf{G}\nu)^q \, d\sigma}{\nu(\Omega)^q}& = \frac{ \sum_{k=1}^{2n}\sigma_k^\frac{1}{1-q} }{ \left(\sum_{k=1}^{2n} \sigma_k^\frac{1}{1-q} \right)^q}\\ 
			& = \left(\sum_{k=1}^{2n} \sigma_k^\frac{1}{1-q} \right)^{1-q}. 
		\end{align*}
	Since $0 < q < 1$, and $\sigma_{2k-1} = a^k$ where $a>1$, the partial sums on the right go to $+\infty$ as $n \to +\infty$, which yields \eqref{inf-sum}. This justifies Remark~\ref{remark7.1}.

	\begin{remark}\label{remark7.2} 	
	The preceding example employs a block matrix kernel which fails to satisfy the weak maximum principle based on a construction with $0$ along the diagonal.
	We have seen that such kernels allow for compact sets $K \in \Omega$ to have infinite capacity, i.e.  $\capa_1 K = + \infty$, which we would like to rule out.
	With this in mind, we can adapt the above construction so that \eqref{inf-sum} holds for a symmetric kernel $G$ such that  $G(x,x) > 0$ for all $x\in \Omega$, i.e., $G$ is strictly positive,  
	but nevertheless the equation $u=\mathbf{G}(u^q d \sigma)$ has a positive 
	solution $u\in L^q(\Omega, \sigma)$. 
	\end{remark} 
	
	Specifically, we adjust each block along the diagonal so that we have kernel $\tilde G$ in place of $G$.
	Let
		\[ \tilde G = \left[ \begin{array}{cc} a & 1 \\ 1 & \frac{1}{a} \end{array} \right] \]
		where $a >0$ is a constant to be specified.
	Note that $\mathbf{G} \nu \le \tilde{\mathbf{G}} \nu$, so we can invoke the above computations to see that \eqref{inf-sum} holds  for $\tilde{\mathbf{G}}$ as well.
	We decompose $\tilde G$ as
		\[ \tilde G = G + G_a = \left[ \begin{array}{cc} 0 & 1 \\ 1 & 0 \end{array} \right] + \left[ \begin{array}{cc} a & 0 \\ 0 & \frac{1}{a} \end{array} \right].  \]
	As shown above, there is a positive solution $u\in L^q(\Omega, \sigma)$ to $u = \mathbf{G}(u^q d \sigma)$.
	By scaling, for $\tilde u = 2^{\frac{1}{1-q}} u$, we have $\frac{1}{2} \tilde u = \mathbf{G}(\tilde u^q d \sigma)$.
	Following the appropriate choice of $a$, we can then ensure that $\frac{1}{2} \tilde u  = \mathbf{G_a}(\tilde u^q d \sigma)$.
	This establishes that $\tilde u$ is a solution, since  we have 
	\[
	\tilde u = \frac{1}{2} \tilde u + \frac{1}{2} \tilde u =  \mathbf{G} (\tilde u^q d \sigma) + \mathbf{G_a}(\tilde u^q d \sigma) = 
	\tilde{\mathbf{G}} (\tilde u^q d \sigma).
	\]   
	The choice of $a$ should be so that
		\begin{align*}
			a \tilde u_1^q \sigma_1 &= \frac{1}{2} \tilde u_1 \\
			\frac{1}{a} \tilde u_2^q \sigma_2 &= \frac{1}{2} \tilde u_2, 
		\end{align*}
		where $a$ is uniquely determined by $a=\left (\frac{\sigma_2}{\sigma_1}\right)^{\frac{1}{1+q}}$. 
		
	With this choice of $a=a_k$ for each block, where $a_k$ depends on $q$ and the values of $\sigma_{2k-1}$ and $ \sigma_{2k}$ defined for the $k$-th block, as specified above.
	Thus, we have a positive solution 
	$\tilde u = \tilde{\mathbf{G}} (\tilde u^q d \sigma)$, where $\tilde u\in L^q(\Omega, \sigma)$, but 
	 \eqref{inf-sum} holds with $\tilde{\mathbf{G}}$ in place of $\mathbf{G}$, which justifies Remark~\ref{remark7.2}. 
	
	The following example shows that the restriction on $q\in (0, q_0]$ where $q_0=\frac{\sqrt{5}-1}{2}$ 
	in  Lemma \ref{lemma2} (a) is sharp. 
	\begin{remark}\label{remark7.3} Let $q \in (q_0,  1)$. 
	Without the weak maximum principle, for a symmetric kernel $G$ there can be a positive solution to $u = \mathbf{G} (u^q d \sigma)$ with $u \in L^q(\Omega, \sigma)$, but 
	$$
	\int_\Omega (\mathbf{G} \sigma)^{\frac{q}{1-q}} d \sigma = +\infty.
	$$ 
	\end{remark}

To construct such an example we employ the above construction of the block matrix kernel $G$ given by \eqref{b-matrix}. Then there exists a positive solution $u \in L^q(\Omega, \sigma)$ to $u = \mathbf{G} (u^q d \sigma)$ 
given by \eqref{u-sol} with finite norm \eqref{u-norm} provided 
$$
\Vert u\Vert^q_{L^q(\sigma)} = \sum_{k=1}^\infty \Big( (\sigma_{2k-1}^q \sigma_{2k})^\frac{q}{1-q^2} \sigma_{2k-1} + (\sigma_{2k-1} \sigma_{2k}^q)^\frac{q}{1-q^2} \sigma_{2k} \Big)<\infty.
$$ 
At the same time we can pick $\sigma_k$ so that, for  $q \in (q_0,  1)$, we have 
$$
\int_\Omega (\mathbf{G} \sigma)^{\frac{q}{1-q}} d \sigma 
=  \sum_{k=1}^\infty (\sigma_{2k-1} \sigma_{2k}^\frac{q}{1-q}  + \sigma_{2k-1}^\frac{q}{1-q}  \sigma_{2k})
 = +\infty.
$$
Indeed, setting $\sigma_{2k-1}=1$ and $\sigma_{2k}=\frac{1}{k}$, we see that 
$$
\Vert u\Vert^q_{L^q(\sigma)} = \sum_{k=1}^\infty (k^{-\frac{q}{1-q^2}} 
 +   k^{-\frac{1}{1-q^2}}) <\infty,
$$ 
since both $\frac{q}{1-q^2} >1$ and $\frac{1}{1-q^2}>1$. On the other hand, 
$$
\int_\Omega (\mathbf{G} \sigma)^{\frac{q}{1-q}} d \sigma 
=  \sum_{k=1}^\infty ( k^{-\frac{q}{1-q}}  +   k^{-1})
 = +\infty.
$$

A slight modification of this example as in Remark~\ref{remark7.2} 
produces a strictly positive kernel $G$ with the same properties. 

	\begin{remark}\label{remark7.4} There are analogous examples that show that the exponents 
	$s=\frac{q}{1-q}$ (for general measures $\sigma$)   and $s=1+q$ (for finite measures $\sigma$) in statements (a) and (b) of 
	Lemma \ref{lemma2}, respectively,  
are sharp as well. We omit the details. 
	\end{remark}


\end{document}